\documentclass[reqno,12pt]{amsart}
\usepackage{amssymb}
\usepackage[english]{babel}
\usepackage{xspace}
\usepackage[bookmarksnumbered,colorlinks]{hyperref}
\usepackage{graphics}
\usepackage{geometry}
\usepackage{enumitem}
\usepackage{mathrsfs}
\renewcommand{\hat}{\widehat}
\newtheorem{theorem}{Theorem}[section]
\newtheorem{lemma}[theorem]{Lemma}
\newtheorem{definition}[theorem]{Definition}
\newtheorem{corollary}[theorem]{Corollary}
\newtheorem{proposition}[theorem]{Proposition}
\theoremstyle{remark}
\newtheorem{remark}[theorem]{Remark}
\theoremstyle{definition}
\newtheorem{example}[theorem]{Example}

\newcommand{\bt}{\begin{theorem}}
	\newcommand{\et}{\end{theorem}}
\newcommand{\bco}{\begin{corollary}}
	\newcommand{\eco}{\end{corollary}}
\newcommand{\bd}{\begin{definition}}
	\newcommand{\ed}{\end{definition}}
\newcommand{\bl}{\begin{lemma}}
	\newcommand{\el}{\end{lemma}}
\newcommand{\bpr}{\begin{proposition}}
	\newcommand{\epr}{\end{proposition}}
\newcommand{\bere}{\begin{remark}}
	\newcommand{\ere}{\end{remark}}
\newcommand{\beq}{\begin{equation}}
	\newcommand{\eeq}{\end{equation}}
\def\bal#1\eal{\begin{align}#1\end{align}}
\def\baln#1\ealn{\begin{align*}#1\end{align*}}
\def\bml#1\eml{\begin{multline}#1\end{multline}}
\def\bmln#1\emln{\begin{multline*}#1\end{multline*}}
\def\bga#1\ega{\begin{gather}#1\end{gather}}
\def\bgan#1\egan{\begin{gather*}#1\end{gather*}}

\DeclareMathOperator{\cat}{cat}

\DeclareMathOperator{\dist}{dist}
\newcommand{\N}{\ensuremath{\mathbb{N}}\xspace}
\newcommand{\R}{\ensuremath{\mathbb{R}}\xspace}
\newcommand{\eps}{\varepsilon}
\newcommand{\vpl}{\varphi_\lambda}

\newcommand{\Fl}{F_{e, \lambda}}
\newcommand{\hl}{h_{e, \lambda}}
\newcommand{\Jl}{E_{e, \lambda}}

\newcommand{\Om}{\Omega_{p,q}(M)}

\title[Multiplicity via localized deformations]{Multiplicity of Finsler geodesics and electromagnetic trajectories via localized deformations}

\author[E. Caponio]{Erasmo Caponio}
\address{Dipartimento di Meccanica, Matematica e Management\hfill\break\indent
	Politecnico di Bari \hfill\break\indent Via Orabona 4,
	70125, Bari, Italy}
\email{erasmo.caponio@poliba.it}

\author[A.V. Germinario]{Anna Valeria Germinario}
\address{Dipartimento  di Matematica\hfill\break\indent
	Universit\`a degli Studi di Bari Aldo Moro\hfill\break\indent Via  Orabona 4,
	70125 Bari, Italy}
\email{anna.germinario@uniba.it}

\author[A. Masiello]{Antonio Masiello}
\address{Dipartimento di Meccanica, Matematica e Management\hfill\break\indent
	Politecnico di Bari \hfill\break\indent Via Orabona 4,
	70125, Bari, Italy}
\email{antonio.masiello@poliba.it}

\begin{document}
\begin{abstract}
Multiplicity of geodesics connecting two points is classically governed by the 	topology of the underlying manifold: the Lusternik--Schnirelmann category of	the path space provides lower bounds for the number of critical points of the	energy functional, and this topological input disappears when the manifold is	contractible.  We show that it can nevertheless be recovered through the	metric itself.  Given a Finsler manifold $(M,F)$ and suitable nested regions	$B_1\subset B_2\subset B_3$, with $M\setminus B_1$ connected and	non-contractible, as in the model case of concentric balls in $\mathbb R^n$,	$n\geq2$, we consider conformal deformations $\varphi_\lambda F$  which can be chosen so as to leave $F$ unchanged outside $B_3$ and grow on $B_2\setminus B_1$.  This growth	creates an  energy barrier, diverging  with $\lambda$, that	low-energy curves cannot cross.  The minimax families can therefore be chosen	in the path space of $M\setminus B_1$, which contains compact subsets of	arbitrarily large category.  Consequently, for every $m\in\mathbb N$ and all	sufficiently large $\lambda$, there exist at least $m$ geodesics joining two	prescribed points, all avoiding $B_1$.  This extends to the Finsler setting an	idea of Capozzi, Fortunato and Greco.
The same scheme applies to Randers metrics $\alpha+\beta$, with $\alpha$ 	deformed and $\beta$ unchanged, as well as  to fixed-energy trajectories of non-relativistic	electromagnetic systems and  relativistic Lorentz-force trajectories, and  light rays in  spacetimes endowed with a causal Killing field.  In the non-relativistic case only the scalar	potential is perturbed, while the magnetic field is preserved.  When the	Killing field is not everywhere timelike, the Fermat principle yields a	possibly singular Randers--Kropina metric on a spacelike hypersurface.  Under a natural admissible category condition, combining  localized barriers	with Randers approximation gives $m$ Randers--Kropina geodesics with pairwise	distinct lengths, equivalently $m$ future-pointing light rays with pairwise distinct arrival times.
\end{abstract}

\keywords{Finsler metric, Randers metric,  Kropina metric, Lusternik--Schnirelmann category, multiplicity of geodesics, electromagnetic trajectories, Lorentz force equation, causal Killing field, stationary spacetimes}
\subjclass[2020]{53C22, 53C50, 53C60, 58E05, 58E10, 70H33, 83C50}
\maketitle

\section{Introduction}
In many variational multiplicity results, the number of critical points is
controlled by the topology of the domain of the functional.  In the
Lusternik--Schnirelmann theory,  the  topological  richness is measured by the category, while the
Palais--Smale condition and the Deformation Lemma provide the analytic
tools which turn category estimates into estimates on the number of critical points.  
In geodesic problems this topological input is usually supplied by the topology
of the underlying manifold $M$.
For instance, if $p,q\in M$, the energy functional
of a Riemannian or Finsler metric is naturally defined on the Hilbert manifold
$\Omega_{p,q}(M)$ of $H^1$-curves joining $p$ to $q$.  Under the 
non-contractibility assumption on $M$, results of Fadell and Husseini imply that
$\Omega_{p,q}(M)$ contains compact subsets of arbitrarily large category. 

If $M$ has trivial topology, in particular on a contractible manifold, there is in general no topological reason for the existence of several geodesics joining two fixed
points.  The purpose of this paper is to show that this obstruction can be
bypassed by a localized deformation of the metric.  The deformation creates a
large energetic barrier around a compact region, so that all curves below a
prescribed energy level are forced to remain in the complement of that region.
Although the ambient manifold may be contractible, this complement can be chosen
to be non-contractible and to carry the  category needed for the
Lusternik--Schnirelmann argument.
 This idea was implemented in~\cite{cfg} in the Riemannian case: even on a
 manifold with trivial topology, for every positive integer $m$ one can modify
 the metric in a suitable localized region,  so as to obtain
 at least $m$ geodesics joining two prescribed points.

Here we extend this result in several directions.  First, we pass from the
Riemannian setting to the Finsler one, treating both conformal deformations of
general (nonreversible) Finsler metrics (Theorem~\ref{main}) and deformations of Randers metrics
$F=\alpha+\beta$ obtained by conformally changing the Riemannian term $\alpha$
while leaving the drift one-form $\beta$ unchanged (Theorem~\ref{mainr}).  Moreover, 
 we show in Section~\ref{R-K} that  the  deformation mechanism producing multiple geodesics can be extended  to a singular Finsler metric which is the pointwise limit of a family of Randers ones. 
Furthermore, an analogous localized perturbation argument  applies  also to non-relativistic 
electromagnetic trajectories with prescribed energy by a perturbation of the scalar potential alone (Theorem~\ref{mainmag}). In the last section we show some applications to  spacetime geometry.
The first concerns charged particles in a standard stationary spacetime endowed
with a stationary electromagnetic field.  For a fixed value of the conserved
energy, the spatial projections of the relativistic Lorentz-force trajectories
satisfy a non-relativistic electromagnetic equation and are, up to
reparametrization, geodesics of an associated Jacobi--Randers metric; see
\cite{cgm}.  The localized deformation of the scalar potential considered
above can then be realized by changing only the scalar coefficient $\phi$ in
the stationary potential vector field $W=\phi\partial_t+A$, while leaving the
spacetime metric, the spatial component $A$, and hence the reduced magnetic
field unchanged.  For the strictly localized data introduced below, the
stationary electromagnetic potential itself is unchanged outside the prescribed
region.  This gives prescribed finite multiplicity for relativistic
Lorentz-force trajectories even when the spatial manifold has trivial topology.
The second application concerns light rays in a standard stationary spacetime.
By the Fermat principle, their spatial projections are pregeodesics of a
Randers metric \cite{cjm,cjs}.  We show that the
localized Finsler deformations introduced above admit natural realizations as
localized deformations of the stationary metric and therefore yield
multiplicity of future- and past-pointing light rays.  In particular, we obtain
a deformation which changes only the spacelike Riemannian metric $g_0$, while
the lapse function $\Lambda$ and the shift one-form $\omega$ are kept fixed.
The third application extends this construction to spacetimes whose 
Killing vector field is  causal.  This is connected to the singular limits of Randers metrics  mentioned above. 
Indeed, the optical geometry of this class of spacetimes is described by
a Randers--Kropina metric: it is Randers where the Killing field is timelike and
becomes Kropina (which is a singular Finsler metric) where the Killing field is lightlike.  This correspondence, as
well as its interpretation in  Zermelo navigation, is developed in
\cite{cjswind}; existence and multiplicity questions for Kropina and
Randers--Kropina geodesics were studied in
\cite{cgmskropina,cjmRK,cjmRKcorr}, while the underlying control-theoretic and
causal-Killing frameworks appear in \cite{cms,bartolocaponio}.  By combining
the methods of those works with the localized
barrier used here, we obtain any prescribed finite number of connecting
Randers--Kropina geodesics, even when the  spacetime itself has trivial
topology, provided that the required non-trivial topology is carried by an
admissible family of curves outside the perturbed region.
There is another  geometric motivation to consider singular Finsler metrics as the Kropina ones, coming  from CR geometry.  The Fefferman metric
admits a distinguished null Killing vector field, namely the infinitesimal
generator of the circle action on the Fefferman bundle.  The chains of a
non-degenerate CR structure are the projections of the null geodesics
transverse to the circle fibers and are locally the geodesics of an associated
Kropina metric \cite{chengchains}.  A survey of the relevant pseudohermitian
and Fefferman geometry, including the earlier proposal to develop a
variational theory of chains, can be found in \cite{dragomir}.  The deformation
considered below produces geodesics of the deformed Kropina metric; these can
be interpreted as chains only if that metric arises from the Fefferman
circle bundle associated with a non-degenerate CR structure.

\section{Localized deformation data and first multiplicity results}

Let $(M,F)$ be a Finsler manifold. We refer to \cite{bcs} for the background material and the properties of Finsler geometry used throughout the paper.

Let  $B_1, B_2$, $B_3$ be subsets of $M$ such that 
\begin{enumerate}[label=(\roman*)]
\item $B_1$ is closed, $B_2$ and $B_3$ are open, $B_1 \subset B_2 \subset B_3$, with $\overline{B_2}\subset B_3$, and
\begin{equation} \label{dab}
\dist(B_1,\partial B_2) > 0,
\end{equation} 
where $\dist$ denotes the (possibly non-symmetric) distance on $M$ induced by $F$ and we set $\dist(B_1,\partial B_2):=\inf\{\dist(x_1,x_2)\mid x_1\in B_1,\ x_2\in \partial B_2\}$;
\item $M\setminus B_1$ is connected and non-contractible, and
$M\setminus B_3$ is a strong deformation retract of $M\setminus B_1$ through a
homotopy
\[
R:(M\setminus B_1)\times[0,1]\to M\setminus B_1
\]
such that each map $R_s:=R(\cdot,s)$ is locally Lipschitz,\footnote{In (ii), local Lipschitz regularity is understood with respect to the
distance of an auxiliary Riemannian metric on $M$. This notion is independent
of the auxiliary metric, since any two Riemannian distances are locally
bi-Lipschitz equivalent. The same auxiliary metric is used to endow the path
space $\Omega_{p,q}(X)$, used in Section~\ref{sectionproofmain}, with its
standard $H^1$ Hilbert-manifold structure.
Condition~\eqref{dab}, however, is stated using the forward distance
induced by the possibly non-reversible Finsler metric $F$.
This is the distance entering the barrier estimate in Lemma~\ref{cl}. If
$\overline{B_2}$ is compact, any Riemannian distance can be used, since on
compact subsets $F$ is uniformly equivalent to any Riemannian norm and $B_1$
and $\partial B_2$ are disjoint compact sets. For non-compact
$\overline{B_2}$, the explicit $F$-distance assumption is essential,
especially when Theorem~\ref{main} is read as a statement about the whole
family $(\varphi_\lambda F)_{\lambda>0}$, fixed independently of the endpoints
$p,q\in M\setminus B_3$.} locally uniformly with respect to $s\in[0,1]$,
and, for all $p, q\in M\setminus B_3$, the induced map
\[
\mathcal R:\Omega_{p,q}(M\setminus B_1)\times[0,1]
\longrightarrow\Omega_{p,q}(M\setminus B_1),
\qquad
\mathcal R(\gamma,s)=R_s\circ\gamma,
\]
is continuous in the $H^1$ topology. Its time-one map is denoted by
$r:=R(\cdot,1)$, so that $r\colon M\setminus B_1\to M\setminus B_3$;
\item there exists a family $(\vpl)_{\lambda>0}$ of smooth functions $\vpl :M  \rightarrow (0, +\infty) $ such that, setting
\begin{equation}\label{infsup}
I_\lambda := \inf_{x \in B_2\setminus B_1} \vpl(x), \qquad S_\lambda :=  \sup_{x \in M\setminus B_3} \vpl(x),
\end{equation}
we have
\begin{equation} \label{hp} 
\lim_{\lambda \rightarrow + \infty} I_\lambda = +\infty \qquad\text{and}\qquad  S  := \sup_{\lambda >0} S_\lambda \in [1,+\infty).
\end{equation}
\item the functions $\vpl$ are normalized from below by
\begin{equation} \label{lowerphi}
\inf_{x\in M}\vpl(x)\geq 1 \qquad\text{for every }\lambda>0.
\end{equation}
\end{enumerate}
\begin{definition}\label{adapted-datum}
A tuple
\[
B_1\subset B_2\subset B_3,\qquad (\varphi_\lambda)_{\lambda>0},
\]
satisfying assumptions {\rm (i)}--{\rm (iv)} will be called a
\emph{localized deformation datum}.  It is called \emph{strictly localized}
if, in addition,
\[
\varphi_\lambda\equiv1\qquad\text{on }M\setminus B_3
\quad\text{for every }\lambda>0.
\]
Given $p,q\in M$, such a datum is said to be \emph{adapted to $p$ and $q$}
if
\[
p,q\in M\setminus B_3.
\]
\end{definition}
\begin{remark}
For a general localized deformation datum, assumptions~{\rm (iii)}-{\rm(iv)} do not
assert that the deformation is neutralized  there.  Statements saying that the metric,
potential, or equation of motion is unchanged outside $B_3$ therefore refer to
strictly localized data.  Remark~\ref{exloc} provides such data.
\end{remark}
 \begin{remark}\label{exloc}
 The smooth functions $\vpl$ satisfying {\rm(iii)} and {\rm(iv)} will be  	used to produce different perturbations of various Finsler metrics, and  	they can be chosen so as to modify the associated dynamical system only  	in an arbitrarily small region. In fact, the whole set of data in  	{\rm(i)}--{\rm(iv)} can be realized explicitly inside a single
 coordinate chart, provided that the complement of the resulting set  	$B_1$ is connected and non-contractible.
 	
 Fix a point $x_0\in M$ and a chart  	$\xi\colon U\to\R^n$, $n=\dim M$, with $\xi(x_0)=0$ and such that  	$\xi(U)$ contains the closed Euclidean ball of some radius $\rho_3>0$  	centered at the origin. We identify $U$ with  	$\xi(U)\subseteq\R^n$ and write $|x|:=|\xi(x)|$ for $x\in U$.  	Fix radii $	0\leq \rho_0<\rho_1<\rho_2<\rho_3$ 	and define $A_i:=\{x\in U:|x|<\rho_i\}$, $i\in\{0,1,2,3\}$. Set $B_0=A_0$, $B_1=\overline{A_1}$,  $B_2=A_2$, $B_3=A_3$,
 so that $	\dist(B_1,\partial B_2)>0$, and {\rm(i)} holds.
 	
 Assume now that $M\setminus B_1$ is connected and non-contractible. 	This is the case, for instance, when $M=\R^n$ with $n\geq2$, which is 	the prototypical situation of a contractible manifold to which our 	results apply. Define
 \[
 R\colon(M\setminus B_1)\times[0,1]\longrightarrow M\setminus B_1, 
 \]
 by
  \[
 R(x,s):=\begin{cases}
 x&x\in M\setminus U \text{ or } |x|\geq \rho_3\\
 \bigl((1-s)|x|+s\rho_3\bigr)\frac{x}{|x|}&	\rho_1<|x|<\rho_3.\end{cases}
 \]
 Then $R$ is continuous, $R(\cdot,0)=\operatorname{id}$, and  	$R(\cdot,s)$ fixes $M\setminus B_3$ pointwise for every 	$s\in[0,1]$. Moreover, $R\bigl(M\setminus B_1,1\bigr)\subset M\setminus B_3$.
 Thus, setting $r:=R(\cdot,1)$, the set $M\setminus B_3$ is a strong
 deformation retract of $M\setminus B_1$. 	
  The maps $R_s:=R(\cdot,s)$ are locally Lipschitz on
 $M\setminus B_1$, locally uniformly with respect to $s\in[0,1]$.
 Moreover, in the coordinates of the chart they can be written as
 \[
 R_s(y)
 =
 \left(|y|+s(\rho_3-|y|)_+\right)\frac{y}{|y|},
 \qquad |y|>\rho_1,
 \]
 where $(a)_+:=\max\{a,0\}$.
 This representation also gives directly the required continuity on the
 $H^1$ path space.  Indeed, the maps $u\to |u|$ and $u\to \frac{u}{|u|}$
 are continuous in $H^1$ on $\{|u|>\rho_1\}$, and so is  the  truncation
 $v\to(\rho_3-v)_+$.    Since multiplication
 \[
 H^1([0,1])\times H^1([0,1],\mathbb R^n)
 \longrightarrow H^1([0,1],\mathbb R^n)
 \]
 is continuous, the above representation implies the continuity of
 $(u,s)\mapsto R_s\circ u$ in the $H^1$ topology and assumption~{\rm(ii)}
 is satisfied (cf. also \cite{musina}).
%
%
%
Finally, choose a smooth function $\eta\colon[0,+\infty)\longrightarrow[0,1]$ which vanishes on a neighborhood of $[0,\rho_0]$, is identically one  	on $[\rho_1,\rho_2]$, and vanishes on $[\rho_3,+\infty)$. Set
 \[
 \vpl(x):=
 \begin{cases}
 	1+\lambda\,\eta(|x|),&x\in U,\\[2pt]
	1,&x\in M\setminus U.
\end{cases}
\]
Since $\eta$ vanishes near $0$ and on $[\rho_3,+\infty)$, each  	$\vpl$ is smooth on $M$. Moreover, it equals $1$ on $B_0$ and on  	$M\setminus B_3$, and it equals $1+\lambda$ on  	$B_2\setminus B_1$. Therefore,
 $ I_\lambda=1+\lambda$, 	$S_\lambda=1$, 	$\vpl\geq1$
 for every $\lambda>0$. Thus 	{\rm(iii)}--{\rm(iv)} are satisfied.
 	
 The resulting datum is strictly localized. In particular, in all the  	theorems below, the perturbation of the corresponding coefficients,  	and hence of the associated equation of motion, is supported in the  	spherical shell
 $ \overline{B_3}\setminus B_0 = 	\{x\in U:\rho_0\leq|x|\leq\rho_3\}$,
 and the equation remains unchanged on  	$(M\setminus B_3)\cup B_0$. Moreover, if the endpoints $p,q\in M$ are prescribed  	in advance, the chart and the radii can be chosen so that
 $p,q\notin B_3$. In that case, the above localized deformation datum
 is adapted to $p$ and $q$ in the sense of
 Definition~\ref{adapted-datum}.
 \end{remark}

As a first application of the deformation data,  we consider  the perturbed  Finsler metrics on $M$, conformal to $F$, defined by $  \vpl F$.
 \begin{remark} \label{compl}
If  $(M,F)$ is forward (resp.\ backward) complete and \eqref{lowerphi} holds, then  the manifold $(M, \vpl F)$ is forward (resp.\ backward) complete as well. Indeed, the $\vpl F$-length of any curve is at least the $F$-length of the same curve, hence every  curve having infinite $F$-length has infinite $\vpl F$-length as well.
\end{remark}
The following multiplicity result for geodesics in $(M, \vpl F)$ holds.
\begin{theorem} \label{main}
Let $(M, F)$ be a (forward or backward) complete Finsler manifold.
Let $p,q\in M$ and let $B_1\subset B_2\subset B_3$, $(\varphi_\lambda)_{\lambda>0}$,
be a localized deformation datum adapted to $p$ and $q$ in the sense of Definition~\ref{adapted-datum}. Then, for every $m\in\mathbb N$, there exists $\lambda_0=\lambda_0(p,q,m)$ such that, for every $\lambda>\lambda_0$, at least $m$ geodesics in $(M,\vpl F)$ joining $p$ to $q$ and contained in $M\setminus B_1$ exist.
\end{theorem}

Our second result concerns Randers metrics. We use the standard notation
\begin{equation} \label{randers}
F=\alpha+\beta,
\qquad
\alpha_x(y):=\sqrt{h_x(y,y)},
\end{equation}
where $h$ is a Riemannian metric on $M$ and $\beta$ is a one-form. This defines a Finsler metric if, for every $x\in M$,
\begin{equation*}
\|\beta\|_{\alpha,x}:=\sup_{v\in T_xM\setminus\{0\}}
\frac{|\beta_x(v)|}{\alpha_x(v)}<1.
\end{equation*}

In the next result, we perturb only the Riemannian metric $h$, equivalently only the Riemannian part $\alpha$ in \eqref{randers}, defining
\begin{equation} \label{hl}
h_\lambda = \vpl h,
\end{equation}
where $\vpl$ satisfies assumptions (iii)--(iv). The corresponding metric is
\begin{equation} \label{pranders}
F_\lambda=\alpha_\lambda+\beta,
\qquad
\alpha_{\lambda,x}(y):=\sqrt{h_\lambda|_x(y,y)}.
\end{equation}
It is again a Randers metric by \eqref{lowerphi}. Indeed,
\[
\|\beta\|_{\alpha_\lambda,x}
:=\sup_{v\in T_xM\setminus\{0\}}
\frac{|\beta_x(v)|}{\alpha_{\lambda,x}(v)}
\leq \|\beta\|_{\alpha,x}<1 .
\]

\begin{remark} \label{rcompl}
If the stronger condition
\begin{equation}\label{norm}
\|\beta\|_\alpha := \sup_{x\in M}\|\beta\|_{\alpha,x} < 1
\end{equation}
holds   and the Riemannian manifold $(M,h)$ is complete, then $(M,F)$ is forward and backward complete as well. Moreover, the same conclusion holds for $(M,F_\lambda)$ with $F_\lambda$ as in \eqref{pranders} whenever \eqref{lowerphi} holds.  Indeed,
\begin{equation*}
\|\beta\|_{\alpha_\lambda} := \sup_{x\in M} \|\beta\|_{\alpha_\lambda,x}
\leq \|\beta\|_\alpha < 1.
\end{equation*}
\end{remark}

\begin{theorem} \label{mainr}
Let $(M,F=\alpha+\beta)$ be a Randers manifold as in \eqref{randers}, where $(M,h)$ is complete and \eqref{norm} holds. Let $p,q\in M$ and let $B_1\subset B_2\subset B_3$, $(\varphi_\lambda)_{\lambda>0}$, be a localized deformation datum adapted to $p$ and $q$ in the sense of Definition~\ref{adapted-datum}. Let $F_\lambda$ be defined as in \eqref{hl}, \eqref{pranders}. Then, for every $m\in\N$, there exists $\lambda_0=\lambda_0(p,q,m)$ such that, for every $\lambda>\lambda_0$,  at least $m$ geodesics in $(M,F_\lambda)$ joining $p$ to $q$ and contained in $M\setminus B_1$ exist.
\end{theorem}

Our third result concerns a classical problem closely related to geodesics 
of Randers metrics. More precisely, we consider the motion of a particle on a connected Riemannian manifold $(M,h)$, endowed with its Levi-Civita connection $\nabla$, under the combined action of a potential $V : M \to \R$ and a magnetic field described by an exact two-form $B=d\beta$ on $M$, where $\beta$ is a one-form. It is well known that the corresponding trajectories satisfy the equation
\begin{equation} \label{mag}
\nabla_{\dot{x}}\dot{x} + \nabla V(x) = -\big(\iota_{\dot x} B\big)^\sharp,
\end{equation}
where $\nabla V$ denotes the gradient of $V$ with respect to $h$, 
$\iota_{\dot x}B := B(\dot x,\cdot)$, and ${}^\sharp:T^*M\to TM$ is the musical isomorphism induced by the metric $h$.

Since the potential is time-independent and $B$ is antisymmetric, the total energy
\begin{equation} \label{toten}
e = \frac{1}{2} h(\dot{x},\dot{x}) + V(x)
\end{equation}
is conserved along every solution $x(s)$ of \eqref{mag}. Solutions of \eqref{mag} connecting two points $p,q\in M$ with prescribed total energy $e$ are naturally related to the Jacobi--Randers metric
\begin{equation}\label{finjac1}
F_e=\alpha_e+\beta,
\qquad
\alpha_{e,x}(y):=\sqrt{h_e|_x(y,y)},
\end{equation}
where
\begin{equation}\label{jac}
h_e:=2(e-V)\,h .
\end{equation}
If $(M,h)$ is complete, then $(M,F_e)$ is a forward and backward complete
Randers manifold provided
\begin{equation} \label{fincond1}
e>\sup_M V
\quad\text{and}\quad
\|\beta\|_{\alpha_e} :=\sup_{v\in TM\setminus\{0\}}
\frac{|\beta(v)|}{\alpha_e(v)}<1. 
\end{equation}
 Indeed, since $e>\sup_M V$, the Jacobi metric $h_e=2(e-V)h$ is bounded below
 by a positive multiple of $h$, hence it is complete whenever $(M,h)$ is
 complete. Moreover, the second condition in \eqref{fincond1} implies both  $F_e$ is really a Finsler metric and its completeness (see \cite[Remark 4.1]{cjm}).

A curve $x$ solves \eqref{mag} with constant energy \eqref{toten} equal to $e$ if and only if it is a pregeodesic of the Randers metric $F_e$ in \eqref{finjac1} (see, e.g., \cite[Appendix A]{cgm}). In the special case $B=0$, this reduces to the classical Maupertuis--Jacobi principle with  the Jacobi metric \eqref{jac}.

Under assumption \eqref{fincond1}, if $M$ is non-contractible, there exist infinitely many trajectories with energy $e$ (see \cite[Proposition 3.1]{cjm}). Here we show that, even without any topological assumption on $M$, a multiplicity result analogous to Theorems \ref{main} and \ref{mainr} can still be obtained by an additive perturbation of the potential. More precisely, for functions $\varphi_\lambda$ as in \eqref{infsup}, \eqref{hp} and \eqref{lowerphi}, we consider
\begin{equation} \label{vl}
V_\lambda := V-\varphi_\lambda +1.
\end{equation}
Then $e-V_\lambda=e-1-V+\varphi_\lambda$,
so the Jacobi conformal factor is increased in the region where $\varphi_\lambda$ is large. 
\begin{theorem} \label{mainmag}
Let $(M,h)$ be a complete Riemannian manifold. Let $V:M\to\R$ be a smooth function such that
\begin{equation} \label{supV}
\overline{V}:=\sup_M V(x)<+\infty
\end{equation}
and let $B$ be an exact two-form on $M$, $B=d\beta$. Assume that \eqref{fincond1} holds for some  $e>\overline V$. Let $p,q\in M$, $p\neq q$, and let $B_1\subset B_2\subset B_3$, $(\varphi_\lambda)_{\lambda>0}$, be a localized deformation datum adapted to $p$ and $q$, with condition \eqref{dab} interpreted using the Riemannian distance $\dist_h$ induced by $h$.
For every $\lambda>0$, let $V_\lambda$ be defined by \eqref{vl}.
Then, for every $m\in\N$, there exists $\lambda_0=\lambda_0(p,q,e,m)$ such that, for every $\lambda>\lambda_0$,  at least $m$ solutions of
\begin{equation} \label{permag}
\nabla_{\dot{x}}\dot{x} + \nabla V_\lambda(x) = -\big(\iota_{\dot x} B\big)^\sharp,
\end{equation} 
joining $p$ to $q$, having energy $e$ and contained in $M\setminus B_1$, exist.
\end{theorem}

\section{Proof of Theorems \ref{main} and \ref{mainr}}\label{sectionproofmain}
Let $X$ be an open subset of $M$ and let $p,q \in X$. We shall consider the infinite dimensional Hilbert manifold
\[\Omega_{p,q}(X) = \left\{x\in H^1 ([0,1], X)  \mid  x(0)=p,\ x(1)=q \right\},\]
where
\[H^1 ([0,1], X) = \left\{ x :[0,1] \rightarrow X  \mid \text{$x$ is absolutely continuous and } \int_0^1 h_R(\dot x,\dot x)\,ds<+\infty\right\},\]
$h_R$ being any complete auxiliary Riemannian metric on $M$ (when dealing with a Randers metric \eqref{randers}, we choose $h_R=h$).   If $X$ is not open,
we shall use the same notation only in the topological sense, namely as the
subspace of $\Omega_{p,q}(M)$ consisting of curves with image contained in $X$.

We consider the energy functionals $J_{\lambda,X}$ and $E_{\lambda,X}$ associated with the metrics $\varphi_\lambda F$ and $F_\lambda$ in~\eqref{pranders}, respectively, defined by
\begin{equation*}
 J_{\lambda,X}(x) = \int_0^1 \varphi_\lambda^2(x)\, F^2(x,\dot{x})\, ds, \qquad\qquad
 E_{\lambda,X}(x) = \int_0^1 F_\lambda^2(x,\dot{x})\, ds,
\end{equation*}
for $x \in \Omega_{p,q}(X)$.
As proved in \cite{cjm}, the critical points of these functionals, for an open subset $X$,  are smooth curves and are precisely the geodesics of $(X,\varphi_\lambda F)$ and $(X,F_\lambda)$, respectively, joining $p$ to $q$.

In what follows, for $c>0$, we denote by $J_{\lambda,X}^c$ and $E_{\lambda,X}^c$ the sublevel sets of $J_{\lambda,X}$ and $E_{\lambda,X}$,  namely
\[J_{\lambda,X}^c=\left\{x \in \Omega_{p,q}(X)\mid J_{\lambda,X}(x)\leq c \right\}, \qquad 
E_{\lambda,X}^c=\left\{x \in \Omega_{p,q}(X) \mid E_{\lambda,X}(x)\leq c\right\}. \]

The proofs of the Theorems~\ref{main} and \ref{mainr} rely on the following lemma.

\begin{lemma} \label{cl}
Let $B_1\subset B_2\subset B_3$, $(\varphi_\lambda)_{\lambda>0}$ be a localized deformation datum adapted to $p,q\in M$. In the case of the functional $E_{\lambda,M}$, assume moreover that $F$ is a Randers metric as in \eqref{randers} satisfying \eqref{norm}. Then there exists a positive  number $c(\lambda)$, $c(\lambda)\to+\infty$ as $\lambda\to+\infty$, such that, if $x\in J_{\lambda,M}^{c(\lambda)}$  (resp. $x\in E_{\lambda,M}^{c(\lambda)}$) then
\begin{equation*}
x([0,1]) \subset M \setminus B_1  .
\end{equation*} 
\end{lemma}

\begin{proof} 
We consider the set $Y$ of the paths  connecting $B_1$ to the boundary of $B_2$ and  parametrized on $[0,1]$, i.e.
\begin{equation} \label{Y}
Y = \left\{ x \in H^1 ([0,1],  \overline{B_2}) \mid   x(0) \in B_1, \, x(1) \in \partial B_2 \right\} .
\end{equation}
 We first examine the functional $J_{\lambda,M}$.  We define the constant $I_1$ as the infimum of the  unperturbed energy functional over $Y$ i.e.
\begin{equation} \label{i}
I_1 = \inf_{ x \in Y} \int_0^1 F^2 (x, \dot{x}) ds.
\end{equation}
By the Cauchy--Schwarz inequality,  for any $x\in Y$ we have 
\[\int_0^1 F^2(x,\dot x) ds\geq\dist(B_1,\partial B_2)^2,\]
thus, by  \eqref{dab},  $I_1>0$.

Now, consider any path $x \in \Om$. If $x([0,1]) \cap B_1 \neq \emptyset$, since $x$ starts and ends outside of $B_3$ (and thus outside of $B_2$), then $x([0,1])$ must have non-empty intersection with   $B_2 \setminus B_1$. Therefore, there exists a subinterval $[s_0, s_1] \subset (0,1)$ such that $x(s_0) \in B_1$, $x(s_1) \in \partial B_2$ and $ x(s) \in B_2 \setminus B_1$ for all $s \in (s_0, s_1)$. We can reparametrize this subpath so that it is defined over $[0,1]$: let $y: [0,1] \rightarrow \overline{B_2}$, $y(s) = x (\alpha (s))$, $s \in [0,1]$, where $\alpha : [0,1] \to [s_0, s_1]$ is defined by $\alpha(s) = (s_1 - s_0)s + s_0$. Then $y \in Y$,  $\dot{y}(s) = (s_1 - s_0)\dot{x}(\alpha(s))$ and
\begin{align}
\int_0^1 F^2(y, \dot{y}) ds & = (s_1 - s_0)^2\int_0^1  F^2( x(\alpha(s)), \dot{x}(\alpha(s))) d s  \nonumber \\
			&= (s_1 - s_0) \int_{s_0}^{s_1} F^2(x, \dot{x})  ds. \label{rep}
		\end{align}
%
We can now obtain a lower bound  for  $J_{\lambda,M} (x)$. Indeed, by the definition of $I_\lambda$ in \eqref{infsup}, \eqref{i}, \eqref{rep} and since $s_1 - s_0 \in (0,1)$ we have
\begin{align} \label{cl1}
J_{\lambda,M} (x) &\geq  \int_{s_0}^{s_1} \varphi^2_\lambda (x )F^2 (x, \dot{x}) ds \nonumber \\
			& \geq I^2_\lambda \int_{s_0}^{s_1} F^2(x, \dot{x}) d s  \nonumber\\ 
			&  = \frac{ I^2_\lambda}{s_1 - s_0} \int_0^1 F^2 (y, \dot{y}) ds  
			 \geq  I^2_\lambda I_1 . 
\end{align}
Thus we can  define 
\[c(\lambda): = \frac{1}{2}  I^2_\lambda I_1. \]
With this choice, if  $x\in \Om$ verifies $J_{\lambda,M}  (x) \leq c(\lambda )$,  then $x$ cannot intersect $B_1$. Indeed, otherwise, 
by \eqref{cl1}, we would have
\[I_\lambda^2 I_1\leq J_{\lambda,M}(x)\leq c(\lambda)= \frac{1}{2} I_\lambda^2I_1,\]
which is  a contradiction. Furthermore, from \eqref{hp}  it follows that $c(\lambda) \to +\infty $ as $\lambda \to +\infty$. 

For the functional $E_{\lambda,M}$, we introduce the constant $I_2$ defined by
\begin{equation} \label{ii}
I_2 = \inf_{ x \in Y} \int_0^1 h (\dot x, \dot{x}) ds.
\end{equation}
Denoting by $\dist_h$ the Riemannian distance induced on $M$ by  $h$, again by the Cauchy--Schwarz inequality,  for any $x\in Y$ we have 
\[\int_0^1h(\dot x,\dot x)\,ds\geq\dist_h(B_1,\partial B_2)^2\geq\frac{\dist(B_1,\partial B_2)^2}{(1+\|\beta\|_\alpha)^2}.\]
thus, by \eqref{dab},   $I_2>0$. Again, consider  $x \in \Om$ such that $x([0,1]) \cap B_1 \neq \emptyset$ and   let $y \in Y$ denote the reparametrization on $[0,1]$ of the subpath of $x$ contained in $B_2 \setminus B_1$ for $s \in (s_0,s_1)$.  Thus
\begin{equation} \label{rep1}
\int_0^1 F_\lambda^2(y, \dot{y}) ds  = (s_1 - s_0) \int_{s_0}^{s_1} F_\lambda ^2(x, \dot{x})  ds.
\end{equation}
By \eqref{norm}, \eqref{ii},  \eqref{rep1},  since $s_1 - s_0 \in (0,1)$ and since $\sqrt{I_\lambda}\geq 1>\|\beta\|_\alpha$ (recall that $\vpl\geq 1$ on $M$, hence $I_\lambda\geq 1$),   we have
\begin{align*}
E_{\lambda,M}(x)  &\geq \int_{s_0}^{s_1} F^2_\lambda( x, \dot{x}) ds  \geq    \int_0^1 F^2_\lambda (y, \dot{y})ds \nonumber \\
 & \geq  \int_0^1 \left( \sqrt{I_\lambda h(\dot y,\dot y)} - \|\beta\|_\alpha\sqrt{h(\dot y,\dot y)} \right)^2 ds  \nonumber\\ 
 &  = \left(\sqrt{I_\lambda } - \|\beta\|_\alpha \right)^2  \int_0^1  h(\dot y,\dot y)  ds  \nonumber \\
	& \geq  \left(\sqrt{I_\lambda } - \|\beta\|_\alpha \right)^2 I_2 .
\end{align*}
Now we can proceed  as in the previous case, by defining 
\[c(\lambda) := \frac{1}{2 }\left(\sqrt{I_\lambda } - \|\beta\|_\alpha \right)^2 I_2. \]
%
%
With this choice, the same contradiction argument used for
$J_{\lambda,M}$ shows that every curve in
$E_{\lambda,M}^{c(\lambda)}$ is contained in $M\setminus B_1$.
Moreover, since $I_\lambda\to+\infty$ by \eqref{hp}, we have
$c(\lambda)\to+\infty$.
\end{proof}
Before proving Theorem~\ref{main} let us 
recall the notion of  category used throughout the paper.  If
$Y\subset X$ are topological spaces, the Lusternik--Schnirelmann category of
$Y$ relative to $X$, denoted by $\cat_X(Y)$,
is the least integer $k\in\mathbb N$ such that $Y$ can be covered by $k$ closed
subsets $Y_1,\ldots,Y_k$ of $Y$, each of which is contractible in $X$.  This
means that, for each $i$, the inclusion $Y_i\hookrightarrow X$ is homotopic in
$X$ to a constant map.  If no such finite covering exists, we set
$\cat_X(Y)=+\infty$.  When $Y=X$, we simply write $\cat(X)$.
Thus, if $Y\subset X\subset Z$, the symbol $\cat_X(Y)$ means that the category
of $Y$ is computed with respect to the ambient space $X$, not with respect to
$Z$.  In general this distinction matters, since a subset may be contractible
in $Z$ without being contractible in $X$.
\begin{proof}[Proof of Theorem \ref{main}]  Fix $p,q\in M$ and $m\in\N$, and let the localized deformation datum be adapted to $p$ and $q$. By Definition~\ref{adapted-datum}, $p,q\in M\setminus B_3$. By assumption (ii), the set
$M \setminus B_3$ is connected and non-contractible, being a strong deformation retract of the connected, non-contractible manifold $M\setminus B_1$.  Let $K$ be a compact subset of $\Omega_{p,q}(M\setminus B_1)$ and  $K':=\mathcal R(K,1)\subset \Omega_{p,q}(M\setminus B_3)$.
Then $K'$ is compact, moreover,
\begin{equation}\label{catineq}
\cat_{\Omega_{p,q}(M\setminus B_1)}(K')
\ge
\cat_{\Omega_{p,q}(M\setminus B_1)}(K).
\end{equation}
Indeed, if $C_1,\ldots,C_N$ is a closed covering of $K'$ by subsets contractible in $\Omega_{p,q}(M\setminus B_1)$,
then
\[
D_i:=K\cap \mathcal R^{-1}(C_i,1)
\]
is a closed covering of $K$. Since $\mathcal R(\cdot, 1)$ is homotopic to the identity on $\Omega_{p,q}(M\setminus B_1)$,
each $D_i$ is contractible in $\Omega_{p,q}(M\setminus B_1)$. Hence
$\cat_{\Omega_{p,q}(M\setminus B_1)}(K)\le N$.
Taking the minimum over all such coverings of $K'$ gives \eqref{catineq}.
By \cite[Proposition~3.2]{fh}, applied to $M\setminus B_1$, the path space
$\Omega_{p,q}(M\setminus B_1)$ contains compact subsets of arbitrarily large
category. Choose then a compact set $K\subset\Omega_{p,q}(M\setminus B_1)$
such that $\cat_{\Omega_{p,q}(M\setminus B_1)}(K)\ge m$.
Then, by \eqref{catineq}, $\cat_{\Omega_{p,q}(M\setminus B_1)}(K')\ge m$.
Moreover, by \eqref{infsup} and \eqref{hp}, for every $x\in K'$ and every
$\lambda>0$,
\begin{equation} \label{at}
	\int_0^1 \vpl^2 (x) F^2 (x, \dot x) ds     \leq  S^2 \cdot  \sup_{y \in K'} \int_0^1 F^2 (y, \dot y) ds  = :C,
\end{equation}  
where $C$ is independent of $\lambda$. Since $c(\lambda)\to+\infty$, choose
$\lambda_0$ such that $c(\lambda)>C$ for every $\lambda>\lambda_0$.
Fix $\lambda>\lambda_0$ and define, for $k=1, \dots, m$,
\[c_k=\inf_{A\in\Gamma_k}\sup_{x\in A}J_{\lambda,M\setminus B_1}(x),\]
where
\[\Gamma_k=\left\{A\subset\Omega_{p,q}(M\setminus B_1): \cat_{\Omega_{p,q}(M\setminus B_1)}(A)\geq k\right\}.\]
Then  we have  $K'\in\Gamma_k$ for every $k=1,\dots,m$ and 
\[0\leq c_1\leq\cdots\leq c_m\leq C<c(\lambda).\]
By Lemma \ref{cl}, for $\varepsilon > 0$ sufficiently small one has $c_k+\varepsilon<c(\lambda)$ and
\[
J^{c_k + \varepsilon}_{\lambda, M}\subset \Omega_{p,q}(M\setminus B_1).
\]
Thus, after identifying this sublevel with its image under the natural inclusion $\Omega_{p,q}(M\setminus B_1)\subset\Omega_{p,q}(M)$, it agrees with $J^{c_k + \varepsilon}_{\lambda, M\setminus B_1}$. Moreover, $J_{\lambda, M }$ verifies the Palais-Smale condition as $(M, \vpl F)$ is (forward or backward) complete, see Remark \ref{compl} and \cite[Theorem 3.1]{cjm}. Since $M\setminus B_1$ is an open subset of $M$, $J_{\lambda, M\setminus B_1} $ and $J_{\lambda, M}$ have the same critical points below the level  $c(\lambda)$. The deformations provided by the classical Deformation Lemma (see for example \cite{pal,rab}) applied to $J_{\lambda, M}$ at levels below $c(\lambda)$ do not increase the value of the functional; hence, they map the sublevel $J^{c(\lambda)}_{\lambda, M}$ into itself and, by Lemma \ref{cl}, they take place inside $\Omega_{p,q}(M\setminus B_1)$, so that they do not decrease the Lusternik--Schnirelmann category relative to $\Omega_{p,q}(M\setminus B_1)$. The standard minimax argument then shows that each $c_k$ is a critical value of $J_{\lambda, M\setminus B_1}$ and that, if $c_k = c_{k+j}$ for some $k$ and some $j\geq 1$, the set of critical points at level $c_k$ has category at least $j+1$ in $\Omega_{p,q}(M\setminus B_1)$, hence it has at least $j+1$ points. In any case, $J_{\lambda, M\setminus B_1}$ has  at least $m$ critical points, which are  geodesics of $(M,  \vpl F)$ joining $p$ to $q$, as $M\setminus B_1$ is an open subset of $M$, and whose support is contained in $M\setminus B_1$. 
\end{proof}

\begin{proof}[Proof of Theorem \ref{mainr}] 
The proof follows the same lines as that of Theorem \ref{main}. We only need to observe that  $(M, F_\lambda)$ is complete (see Remark \ref{rcompl}). Moreover,  an estimate similar to \eqref{at} holds in this case as well, since  for any $\lambda >0$ and $x\in K'$, $K'$ compact subset of $\Omega_{p,q}(M\setminus B_3)$,  by \eqref{infsup}, \eqref{hp} and \eqref{norm}, recalling that $x$ takes values in $M\setminus B_3$, where $\vpl\leq S$, we have
\begin{align*}  
\int_0^1 F^2_\lambda (x, \dot x) ds    &  =   \int_0^1 \left( \sqrt{\vpl   h(\dot x,\dot x)}+\beta(\dot x ) \right)^2 ds   \\
 &\leq  \int_0^1 \left( \sqrt{  S h(\dot x,\dot x)}+ \|\beta\|_\alpha \sqrt{h(\dot x,\dot x)} \right)^2 ds   \\
& \leq  (\sqrt{S}+ \|\beta\|_\alpha)^2    \sup_{y \in K'} \int_0^1 h (\dot y, \dot y) ds  =: C.
\end{align*} 
\end{proof}
\section{Proof of Theorem \ref{mainmag}}
Fix $p,q\in M$, $m\in\N$, and a localized deformation datum adapted to $p$ and $q$. Fix $e > \overline V$. The solutions of \eqref{permag} with total  energy  $e$ are pre-geodesics of
\begin{equation}\label{finjac2}
\Fl=\alpha_{e,\lambda}+\beta,\qquad
\alpha_{e,\lambda,x}(y):=\sqrt{\hl|_x(y,y)},\qquad
\hl:=2(e-V_\lambda)\,h,
\end{equation}
where $V_\lambda$ is defined in \eqref{vl}. This is  a forward and backward complete Randers metric on $M$. Indeed, by \eqref{supV} and \eqref{lowerphi},
\[
e-V_\lambda=e-1-V+\varphi_\lambda\ge e-\overline V>0.
\]
Thus, $\hl\geq  2(e-\overline V)h$ and since $(M,h)$ is complete,  $(M,\hl)$ is complete as well. Moreover, by \eqref{fincond1},
\begin{equation} \label{fincond2}
\|\beta\|_{\alpha_{e,\lambda}} := \sup_{v\in TM\setminus\{0\}} \frac{|\beta(v)|}{\alpha_{e,\lambda}(v)} \leq  \|\beta\|_{\alpha_e} < 1.
\end{equation}
We consider the energy functional associated with  \eqref{finjac2},
\begin{equation*}
\Jl (x)= \int_0^1 \Fl^2 (x, \dot x)ds, \qquad   x\in \Omega_{p,q}(M)
\end{equation*}
and prove the following lemma analogous to Lemma~\ref{cl}.
\begin{lemma} \label{clj}
There exists a positive number $c(\lambda)$, with $c(\lambda)\to+\infty$ as $\lambda\to+\infty$, such that, if $x\in\Jl^{c(\lambda)}$ then
\begin{equation*}
x([0,1]) \subset M \setminus B_1  .
\end{equation*} 
\end{lemma}
\begin{proof}
Define the set $Y$ as in \eqref{Y} and observe that, by \eqref{dab}, for any $y\in Y$
\begin{equation} \label{v}
\int_0^1 h (\dot y, \dot y)\, ds \geq \dist_h (B_1, \partial B_2)^2 >0.
\end{equation}
We obtain an estimate analogous to \eqref{cl1}. To this end,  for any $x \in \Om$ with  $x([0,1]) \cap B_1 \neq \emptyset$, let  $[s_0, s_1] \subset (0,1)$ be  an interval such that $x(s_0) \in B_1$, $x(s_1) \in \partial B_2$ and $x((s_0,s_1))\subset B_2\setminus B_1$.   Let  $y$ be a reparametrization of $x|_{[s_0,s_1]}$ on the interval $[0,1]$; by \eqref{finjac2} and \eqref{fincond2},   we have
\begin{align} \label{cl3}
\Jl(x)  &\geq \int_{s_0}^{s_1}\Fl^2( x, \dot{x}) ds    = \frac{1}{s_1 - s_0} \int_0^1 \Fl^2 (y, \dot{y})\, ds  \nonumber \\
			& \geq 2 (1- \|\beta\|_{\alpha_e})^2   \int_0^1 (e-V_\lambda (y))\,h( \dot y, \dot{y}) d s . 
\end{align}
By \eqref{infsup}, \eqref{hp} and \eqref{supV}, on $B_2\setminus B_1$ we have
\[
 e-V_\lambda(y)
 \ge e-1-\overline V+I_\lambda .
\]
Thus, by  \eqref{v} and \eqref{cl3},
\[
\Jl(x)  \geq 2 (1- \|\beta\|_{\alpha_e})^2  (e-1-\overline V+I_\lambda)\, \dist_h(B_1, \partial B_2)^2 .
\]
Then for the function
\[
c(\lambda ) =  (1- \|\beta\|_{\alpha_e})^2  (e-1-\overline V+I_\lambda)\, \dist_h (B_1, \partial B_2)^2  ,
\]
we get the thesis.

\end{proof}
\begin{proof}[Proof of Theorem \ref{mainmag}] 
By Lemma \ref{clj}, the proof follows the same lines as that of Theorem \ref{main}. We only have  to establish an estimate analogous to \eqref{at}. To this aim, let $K'$ be a compact subset of $\Omega_{p,q}(M\setminus B_3)$. For every $\lambda >0$ and $x\in K'$, by \eqref{infsup} and \eqref{hp} we have $\vpl \leq S$ on $M\setminus B_3$; since $x$ takes values in $M\setminus B_3$, it follows that
\[
e-V_\lambda(x)\le e-1-V(x)+S.
\]
Moreover, by \eqref{fincond2},
\[
 \left|\beta(\dot{x})\right|    \leq   \sqrt{2(e-V_\lambda (x))   h  (\dot x , \dot x ) }
 \leq   \sqrt{2(e-1-V(x)+S)   h  (\dot x , \dot x ) } .
\]
Therefore
\begin{align*}  
\int_0^1 \Fl^2  (x, \dot x) ds    &  \leq 8 \int_0^1  (e-1-V(x)+S )h  (\dot x , \dot x )ds     
     \leq  8     \sup_{y \in K'} \int_0^1  (e-1-V(y)+S )h  (\dot y , \dot y )ds . 
\end{align*} 
The last quantity is finite because the compactness of $K'$ in $\Omega_{p,q}(M\setminus B_3)$ implies that the union of the images of the curves in $K'$ is a compact subset of $M$, and $V$ is bounded on this compact set.
The minimax argument therefore gives at least $m$ geodesics of the
Jacobi--Randers metric $F_{e,\lambda}$ contained in $M\setminus B_1$. After a reparametrization, these geodesics yield solutions of
\eqref{permag} with total energy $e$. Since reparametrization does not change
the image, the corresponding trajectories are still contained in
$M\setminus B_1$.
\end{proof}
\section{Applications to electromagnetic trajectories and null geodesics in spacetimes with a causal Killing field}
\subsection{Localized deformations of the scalar potential in a standard stationary spacetime endowed with an electromagnetic field}
In this subsection we explain how the preceding multiplicity results apply to
the Lorentz-force equation in a stationary spacetime; see \cite{cgm}.
Let $M_0$ be a smooth, connected manifold and consider the product spacetime
$M=\mathbb{R}\times M_0$; let $t$ denote the coordinate on the
$\mathbb{R}$ factor.  A {\em standard stationary} Lorentzian metric on $M$ has the
form
\begin{equation}\label{ssst}
	g=-\Lambda(x)\,dt^2
	+\omega_x\otimes dt+dt\otimes\omega_x
	+g_0|_x,
	\qquad x\in M_0,
\end{equation}
where $g_0$ is a Riemannian metric on $M_0$,
$\omega\in\Omega^1(M_0)$, and
$\Lambda\in C^\infty(M_0)$ is positive.  All these objects are understood to
be pulled back to $\mathbb{R}\times M_0$ and are independent of $t$.
Consequently, $K=\partial_t$ is a timelike Killing vector field and
$g(K,K)=-\Lambda$.
It is convenient to introduce the  one-form and Riemannian tensors
\[
\widehat\omega=\frac{\omega}{\Lambda},
\qquad
{\widehat g}_0=\frac{g_0}{\Lambda},
\qquad
h_0=g_0+\frac{1}{\Lambda}\omega\otimes\omega,
\]
and
\[\widehat h=\frac{1}{\Lambda}h_0
={\widehat g}_0+\widehat\omega\otimes\widehat\omega;
\]
in particular, $h_0$ and $\widehat h$ are Riemannian metrics on $M_0$.
Then we can rewrite $g$ in \eqref{ssst} as 
\[
g=-\Lambda(dt-\widehat\omega)^2+h_0 .
\]
Let $\mathcal A$ be a stationary electromagnetic potential,
$d\mathcal A=\mathcal F$, and fix a charge-to-mass ratio $\rho$.
Writing $\mathcal A=g(W,\cdot)$ with $W=\phi\partial_t+A$, set
\[
\chi=\phi-\widehat\omega(A).
\]
A worldline  $\gamma=(t,x)$ (i.e. a smooth curve parametrized with proper time, so  that $g(\dot\gamma, \dot\gamma)=-1$) is a solution of the Lorentz-force equation for a charge-to-mass ratio $\rho$,  if $\gamma$  solves
\begin{equation*}
\nabla_{\dot\gamma}\dot\gamma=-\rho\,(\iota_{\dot\gamma}\mathcal F)^\sharp,
\end{equation*}
where $\nabla$ is the Levi-Civita connection of $g$, $\iota_{\dot\gamma}\mathcal  F:=\mathcal F(\dot\gamma,\cdot)$ and ${}^\sharp$ is the musical isomorphism ${}^\sharp:T^*M\to TM$ associated with $g$.
The conserved energy (per unit-mass)  of a solution is
\[
\varepsilon
=
-g(\partial_t,\dot\gamma)-\rho\mathcal A(\partial_t)
=
\Lambda\big(\dot t-\widehat\omega(\dot x)+\rho\chi(x)\big).
\]

By \cite[Theorem 5.1]{cgm}, the projection $x$ of any such solution solves the non-relativistic electromagnetic
system on $(M_0,h_0)$
\[
\nabla^{h_0}_{\dot x}\dot x+\nabla^{h_0}V(x)
=
-\big(\iota_{\dot x}\mathcal B\big)^{\sharp,h_0},
\qquad
h_0(\dot x,\dot x)+2V(x)=-1,
\]
where
\[
V=-\frac{(\varepsilon-\rho\Lambda\chi)^2}{2\Lambda},
\qquad
\Theta=\varepsilon\widehat\omega+\rho\mathcal A_{M_0},
\qquad
\mathcal B=d\Theta,
\]
and $\mathcal A_{M_0}=A^{\flat,h_0}$. Conversely, once $x$ is known, the time
component is recovered from
\[
\dot t=\frac{\varepsilon}{\Lambda}+\widehat\omega(\dot x)-\rho\chi(x).
\]
Hence, the search for solutions  of the Lorentz-force equation with fixed energy $\varepsilon$ reduces to the search for geodesics of a Jacobi-Randers metric. With $\gamma$ parametrized with proper time, this Jacobi--Randers metric is
$F=F_{-1/2}$, i.e. 
\[
F(x,y)=\sqrt{(-1-2V(x))\,h_0(y,y)}+\Theta_x(y),
\]
on the domain where $-1>2V$ and $\|\Theta\|_{(-1-2V)h_0}<1$. In particular, if
$(M_0,h_0)$ is complete and
\begin{equation}\label{infandsup}
\inf_{M_0}\frac{(\varepsilon-\rho\Lambda\chi)^2}{\Lambda}>1,
\qquad
\sup_{v\in T{M_0}\setminus\{0\}}
\frac{|\Theta(v)|}{\sqrt{(-1-2V)h_0(v,v)}}<1,
\end{equation}
then $F$ is a forward and backward complete Randers metric on ${M_0}$.

Theorems~\ref{main} and \ref{mainr} act on the  
metric $F$: the former rescales $F$ conformally to
$F_{\lambda}=\varphi_\lambda F$, changing both the Riemannian part
$\sqrt{2(e-V)h_0}$ and the one-form $\Theta$, while the latter deforms only
the Riemannian part, keeping $\Theta$ fixed. In  these two results the multiplicity is intrinsic to the reduced Finsler
problem on $M_0$, and their  interpretation  as a multiplicity result for
electromagnetic trajectories with fixed  Killing energy   holds only when such
deformed metrics can be realized as  Jacobi-Randers metrics associated with  a
standard stationary spacetime and  a stationary electromagnetic field.

Theorem~\ref{mainmag}, by contrast, admits a direct electromagnetic
realization when $\rho\neq0$. It keeps the magnetic field $\mathcal B=d\Theta$
fixed and replaces the potential by $V_\lambda=V-\varphi_\lambda+1$. 
Thus, one can produce $V_\lambda$ by modifying only the scalar coefficient
$\phi$ in the stationary potential vector field $W=\phi\partial_t+A$, while
leaving $\Lambda$, $\widehat\omega$ and $A$, hence $\Theta$ and
$\mathcal B$, unchanged.
Namely, set
\[
Z:=\varepsilon-\rho\Lambda\chi=\varepsilon-\rho\Lambda(\phi-\widehat\omega(A)).
\]
Solving 
\[
\bigl(\varepsilon-\rho\Lambda(\phi_\lambda-\widehat\omega(A))\bigr)^2
=
Z^2+2\Lambda(\varphi_\lambda-1)
\]
gives
\[
\phi_\lambda^\sigma
=
\widehat\omega(A)
+
\frac{\varepsilon-\sigma
	\sqrt{Z^2+2\Lambda(\varphi_\lambda-1)}}{\rho\Lambda},
\qquad \sigma\in\{+1,-1\}.
\]
On the region where $\varphi_\lambda=1$, this reduces to
\begin{equation}\label{phisigma}
\phi_\lambda^\sigma=\widehat\omega(A)+\frac{\varepsilon-\sigma |Z|}{\rho\Lambda}.
\end{equation}
Thus $\phi_\lambda=\phi$ there provided $\sigma |Z|=Z$, i.e. provided
$\sigma=\operatorname{sign}Z$. In particular, a single smooth branch can be
chosen if $Z=\varepsilon-\rho\Lambda\chi$ has constant sign on the
region where the perturbation is equal to $1$, as it holds in the following proposition. 
\begin{proposition}
Assume that $\rho\ne0$, that $(M_0,h_0)$ is complete, and that \eqref{infandsup} holds.
Let $p,q\in M_0$, $p\ne q$, and let
$B_1\subset B_2\subset B_3$, $(\varphi_\lambda)_{\lambda>0}$ be a localized
deformation datum adapted to $p$ and $q$, with condition~\eqref{dab}
interpreted using the distance of $h_0$.  Choose
$\sigma=\operatorname{sign}Z$ in \eqref{phisigma}.
Then, for every $m\in\mathbb N$ and any $t_0\in\R$, there exists
$\lambda_0=\lambda_0(p,q,\varepsilon,m)>0$ such that, for every
$\lambda>\lambda_0$, the stationary electromagnetic potential vector field
$W_\lambda=\phi_\lambda^\sigma\partial_t+A$
gives at least $m$ solutions to the corresponding Lorentz-force equation,  with conserved
energy $\varepsilon$, starting at $(t_0,p)$ and ending on the stationary line
$\mathbb R\times\{q\}$, whose spatial projections are contained in
$M_0\setminus B_1$.    If the deformation datum is strictly localized, then $W_\lambda=W$, and hence the
stationary electromagnetic potential and field are unchanged, on
$\mathbb R\times(M_0\setminus B_3)$.
\end{proposition}
\begin{proof}
For the reduced system, take $e=-1/2$, $h=h_0$, $\beta=\Theta$ and
$B=d\Theta$ in Theorem~\ref{mainmag}.   Notice that from the first inequality in \eqref{infandsup}, $Z$ must have constant sign on $M_0$. Thus,
$\phi_\lambda^\sigma$ corresponds to the change of the reduced scalar potential from $V$
to $V_\lambda=V-\varphi_\lambda+1$,  while leaving $\Theta$ fixed. 
From \eqref{infandsup} and  Theorem
\ref{mainmag} we then get at least  $m$ solutions of the reduced system,  with  energy
$-1/2$,  as reparametrization of geodesics of the associated backward and forward complete Randers metric.   The fixed-energy correspondence of \cite[Theorem~5.1]{cgm}, together
with
\[
\dot t=\frac{\varepsilon}{\Lambda}+\widehat\omega(\dot x)-\rho\chi_\lambda,\qquad \chi_\lambda=\phi_\lambda^\sigma-\hat\omega(A),
\]
lifts them to the asserted relativistic trajectories.  For strictly localized
data, $\varphi_\lambda=1$ on $M_0\setminus B_3$ and 
$\phi_\lambda^\sigma=\phi$ there.
\end{proof}

\subsection{Localized deformations of the spacelike metric in standard stationary spacetimes and multiplicity of null geodesics}

For null geodesics of $g$ in \eqref{ssst} it is well known that the projected curves are pregeodesics of the Randers     metrics, called  {\em Fermat metrics}  (see \cite{cjm,cjs})
\[
F_\pm(x,y)=\sqrt{\widehat h_x(y,y)}\pm\widehat\omega_x(y),
\]
where the sign $\pm$ is chosen according to the time orientation of the null geodesics.
In this case the geometric deformations of Theorems~\ref{main} and
\ref{mainr} have two different spacetime realizations.

Theorem~\ref{main} gives the full conformal Finsler deformation
\[
F_{\pm,\lambda}^{C}
=
\varphi_\lambda F_\pm
=
\sqrt{\varphi_\lambda^2\widehat h_x(y,y)}
\pm
\varphi_\lambda\widehat\omega_x(y).
\]
This is realized as the pair of Fermat metrics of
\[
g_\lambda^{C}
=
-\Lambda(dt-\varphi_\lambda\widehat\omega)^2
+
\varphi_\lambda^2h_0 .
\]
Indeed,
\[
\sqrt{\frac{\varphi_\lambda^2h_0}{\Lambda}}
\pm
\varphi_\lambda\widehat\omega
=
\sqrt{\varphi_\lambda^2\widehat h}
\pm
\varphi_\lambda\widehat\omega .
\]

For Theorem~\ref{mainr}, the deformation is
\[
F_{\pm,\lambda}^{R}
=
\sqrt{\varphi_\lambda\widehat h_x(y,y)}
\pm
\widehat\omega_x(y).
\]
It is realized by the stationary metric
\[
g_\lambda^{R}
=
-\Lambda dt^2+\omega\otimes dt+dt\otimes \omega +g_{0,\lambda},
\qquad
g_{0,\lambda}
=
\varphi_\lambda g_0+
\frac{\varphi_\lambda-1}{\Lambda}\omega\otimes\omega .
\]
Equivalently,
\[
g_\lambda^{R}
=
-\Lambda(dt-\widehat\omega)^2+\varphi_\lambda h_0 .
\]
Indeed, the associated Fermat metrics are
\[
\sqrt{
	\frac{1}{\Lambda}g_{0,\lambda}
	+
	\widehat\omega\otimes\widehat\omega
}
\pm
\widehat\omega
=
\sqrt{\varphi_\lambda\widehat h}
\pm
\widehat\omega .
\]
If the deformation datum is strictly localized, both $g_\lambda^{C}$ and
$g_\lambda^{R}$ coincide with $g$ on
$\mathbb R\times(M_0\setminus B_3)$.  Theorem~\ref{main} yields
multiplicity of future-pointing (and past-pointing) light rays connecting an event $(t_0, x_0)$ with a line $t\mapsto (t,x_1)$ for the localized stationary perturbation
$g_\lambda^{C}$, whereas Theorem~\ref{mainr} yields multiplicity of the same type of light rays
for the different localized stationary perturbation $g_\lambda^{R}$, provided
the completeness assumptions required in the two theorems hold.

We now point out that the localization method also gives a natural deformation
result for light rays in standard stationary spacetimes when only the
Riemannian metric $g_0$ on ${M_0}$ is deformed.
Let $p,q\in {M_0}$ and let $B_1\subset B_2\subset B_3$, $(\varphi_\lambda)_{\lambda>0}$,
be a localized deformation datum adapted to $p$ and $q$.
Let us assume that
\begin{equation}\label{dg0M12}
d_{\widehat g_0}(B_1,\partial B_2)>0, \qquad
M_{12}:=\sup_{x\in B_2\setminus B_1}\|\widehat\omega\|_{\widehat g_0,x}<+\infty.
\end{equation}
The latter condition is automatic when $\overline{B_2}$ is compact.
For every $\lambda>0$ define a new standard stationary metric on
$\mathbb R\times {M_0}$ by
\begin{equation}\label{glambda}
g_\lambda=-\Lambda\,dt^2+\omega\otimes dt+dt\otimes\omega+\varphi_\lambda g_0.
\end{equation}
Thus only the spacelike Riemannian metric $g_0$ is deformed, while the lapse
$\Lambda$ and the shift one-form $\omega$ are kept fixed.
The Fermat metrics associated with $g_\lambda$ are
\[
F_\lambda^\pm(x,y)
=
\sqrt{\widehat h_{\lambda,x}(y,y)}
\pm \widehat\omega_x(y),
\]
where
\[
\widehat h_\lambda
= \widehat\omega\otimes\widehat\omega +
\varphi_\lambda\widehat g_0 .
\]
Notice that, since $\varphi_\lambda\ge1$,
\begin{equation}\label{Flambdabb}
F_\lambda^\pm(x,y)\ge F^\pm(x,y).
\end{equation}
Consequently, every forward or backward completeness property of $F^\pm$ is
inherited by $F_\lambda^\pm$. Notice also that
$F^-(x,y)=F^+(x,-y)$, so forward completeness of one of the two Fermat metrics
is equivalent to backward completeness of the other one.
\begin{remark}
By \cite[Theorem 4.4]{cjs},  forward and backward completeness of $F^+$ (and then of $F^-$ as well) is equivalent to the property that for some $\bar t\in \R$ (and then for all $\bar t\in \R$), $\{\bar t\}\times {M_0}$ is a Cauchy hypersurface of the standard stationary spacetime $(M,g)$.
A weaker notion that can replace forward or backward completeness of $F^+$ in the next proposition is \eqref{star} introduced below.
\end{remark}

\begin{proposition}\label{lightdeformation}
 Let $(\mathbb R\times {M_0},g)$ be a standard stationary spacetime as in \eqref{ssst} and
 assume that the Fermat metric
 \[
 F^+(x,y)=\sqrt{\widehat h_x(y,y)}+\widehat\omega_x(y)
 \]
 is forward or backward complete. Let $p,q\in {M_0}$ and let  $ B_1\subset B_2\subset B_3$,   $(\varphi_\lambda)_{\lambda>0}$,  be a localized deformation datum adapted to $p$ and $q$ such that  \eqref{dg0M12} holds.
 Then, for every $m\in\mathbb N$ and any $t_0\in\R$, there exists   $\lambda_0=\lambda_0(p,q,m)>0$ such that, for every     $\lambda>\lambda_0$, the spacetime $(\mathbb R\times {M_0},g_\lambda)$, $g_\lambda$ in \eqref{glambda}, admits at     least $m$ future-pointing light rays joining $(t_0,p)$ to the line     $\mathbb R\times\{q\}$ and whose spatial projections are contained in     ${M_0}\setminus B_1$. An analogous statement holds for past-pointing light rays.
\end{proposition}
\begin{proof}
 We prove the future-pointing case. By Fermat's principle, future-pointing     light rays joining $(t_0,p)$ to the line $\mathbb R\times\{q\}$ are in     one-to-one correspondence, up to orientation-preserving reparametrization,     with geodesics of the Fermat metric $F_\lambda^+$ joining $p$ to $q$.   From \eqref{Flambdabb}, the metric $F_\lambda^+$ is forward or backward     complete. Hence its energy functional satisfies the Palais--Smale condition, by     \cite[Theorem~3.1]{cjm}.
 We first establish the localization estimate. Let $x\in\Omega_{p,q}({M_0})$ and    assume that $x([0,1])$ intersects $B_1$. Since $p,q\in {M_0}\setminus B_3$, the    curve must cross $B_2\setminus B_1$. Thus there is a subinterval     $[s_0,s_1]\subset(0,1)$ such that, after reparametrization, we obtain a curve     $y\in H^1([0,1],\overline B_2)$ with
 \[
 y(0)\in B_1,\qquad y(1)\in\partial B_2,  \qquad    y(s)\in B_2\setminus B_1\quad\text{for }s\in(0,1).
 \]
 On $B_2\setminus B_1$ we have $\varphi_\lambda\ge I_\lambda$. Therefore,    for every $v\in T(B_2\setminus B_1)$, it holds:
\begin{align*}
F_\lambda^+(v)&=    \sqrt{\widehat\omega(v)^2    +\varphi_\lambda\widehat g_0(v,v)}    +\widehat\omega(v)\\
    &\ge    \sqrt{\widehat\omega(v)^2    +I_\lambda\widehat g_0(v,v)}    -|\widehat\omega(v)|\\
    &\ge    \bigl(\sqrt{I_\lambda+M_{12}^2}-M_{12}\bigr)    \sqrt{\widehat g_0(v,v)}.
\end{align*}
It follows, arguing as in Lemma~\ref{cl}, that
\[
E_\lambda(x)    :=    \int_0^1(F_\lambda^+)^2(x,\dot x)\,ds    \ge    \bigl(\sqrt{I_\lambda+M_{12}^2}-M_{12}\bigr)^2    d_{\widehat g_0}(B_1,\partial B_2)^2.
 \]
Thus, setting
 \[
 c(\lambda)    = \frac{1}{2}   \bigl(\sqrt{I_\lambda+M_{12}^2}-M_{12}\bigr)^2    d_{\widehat g_0}(B_1,\partial B_2)^2,
 \]
 we have $c(\lambda)\to+\infty$, and every curve in the sublevel    $E_\lambda^{c(\lambda)}$ is contained in ${M_0}\setminus B_1$.
    
 It remains to obtain a uniform upper bound on compact sets contained in  $\Omega_{p,q}({M_0}\setminus B_3)$. Let $K\subset\Omega_{p,q}({M_0}\setminus B_3)$    be compact. Since $\varphi_\lambda\le S$ on ${M_0}\setminus B_3$ and $S\ge1$,
 we have
 \[
 \widehat h_\lambda    =    \widehat\omega\otimes\widehat\omega    +\varphi_\lambda\widehat g_0
 \le     S\widehat h    \quad\text{on }{M_0}\setminus B_3.
 \]
 Moreover, $|\widehat\omega(v)|\le\sqrt{\widehat h(v,v)}$ for every    $v\in T{M_0}$. Therefore, for $x\in K$,
\[
 F_\lambda^+(x,\dot x)    \le    (\sqrt S+1)\sqrt{\widehat h(\dot x,\dot x)},
\]
and consequently
\[
\int_0^1(F_\lambda^+)^2(x,\dot x)\,ds    \le    (\sqrt S+1)^2    \sup_{y\in K}\int_0^1\widehat h(\dot y,\dot y)\,ds    =:C,
\]
where $C$ is independent of $\lambda$. The usual Lusternik--Schnirelmann argument now applies exactly as in the proof of Theorem~\ref{mainr}. The compact set of large category is first   obtained in $\Omega_{p,q}({M_0}\setminus B_1)$ and then pushed, by the retraction     associated with the localized deformation datum, into     $\Omega_{p,q}({M_0}\setminus B_3)$ without decreasing its category. Choosing $\lambda$ so large that     $c(\lambda)>C$, the minimax levels lie below $c(\lambda)$, and the deformation     lemma takes place inside $\Omega_{p,q}({M_0}\setminus B_1)$. Hence the energy     functional of $F_\lambda^+$ has at least $m$ critical points, which are     geodesics of $F_\lambda^+$ joining $p$ to $q$ and contained in     ${M_0}\setminus B_1$. By Fermat's principle, these geodesics lift to     future-pointing light rays of $(\mathbb R\times {M_0},g_\lambda)$ joining     $(t_0,p)$ to the line $\mathbb R\times\{q\}$.
    
For the past-pointing case one uses $F_\lambda^-$. Since     $F^-(x,v)=F^+(x,-v)$, the assumed one-sided completeness of $F^+$ gives the     opposite one-sided completeness of $F^-$. Furthermore,    $F_\lambda^-\ge F^-$, and the same barrier and upper estimates hold with     $F_\lambda^-$ in place of $F_\lambda^+$.
\end{proof}

\subsection{Localized deformations of Randers--Kropina metrics}\label{R-K}
We finally consider the singular optical geometry which occurs when the
Killing vector field $\partial_t$ is  causal.  Let ${M_0}$ be a connected smooth
manifold, let $g_0$ be a Riemannian metric on ${M_0}$, let $\omega$ be a  one-form on $M_0$, and let $\Lambda\colon {M_0}\to[0,+\infty)$ be smooth.  On $\mathbb R\times {M_0}$ set
\begin{equation}\label{causalpartialt}
 g=-\Lambda\,dt^2+dt\otimes\omega+\omega\otimes dt+g_0,
 \qquad
 \Lambda+\|\omega\|_{g_0}^2>0.
\end{equation}
Then $\partial_t$ is a complete causal Killing vector field, timelike on
$\{\Lambda>0\}$ and lightlike on $\{\Lambda=0\}$.  The future optical
Lagrangian on ${M_0}$ is the {\em Randers--Kropina metric}
\begin{equation}\label{RKmetric}
 F(x,v)=
 \frac{g_0(v,v)}{-\omega(v)+
 \sqrt{\Lambda(x)g_0(v,v)+\omega(v)^2}}.
\end{equation}
Its conic domain is
\begin{equation}\label{RKdomain}
 \mathcal A_x=
 \begin{cases}
 T_x{M_0}\setminus\{0\},&\Lambda(x)>0,\\[2pt]
 \{v\in T_x{M_0}:\omega(v)<0\},&\Lambda(x)=0.
 \end{cases}
\end{equation}
At a point $x$ where $\Lambda(x)=0$, formula~\eqref{RKmetric} reduces to 
\begin{equation*}
 F(x,v)=-\frac{g_0(v,v)}{2\omega(v)},
 \qquad \omega(v)<0,
\end{equation*}
which is a singular Minkowski norm on $T_xM_0$ called {\em Kropina norm}.
The Fermat principle identifies the pregeodesics of $F$ with the spatial
projections of the  non-vertical future-pointing lightlike geodesics of
\eqref{causalpartialt}, see \cite[Theorem 5.5]{cjswind}.  

Let $p,q\in {M_0}$, $p\ne q$, and let
$B_1\subset B_2\subset B_3$, $(\varphi_\lambda)_{\lambda>0}$ be a localized
deformation datum adapted to $p$ and $q$.  In this subsection we assume in
addition that $\overline{B_2}$ is compact and that the deformation is pointwise
monotone, namely 
\begin{equation}\label{monotone-localized}
 \lambda_2\ge\lambda_1
 \quad\Longrightarrow\quad
 \varphi_{\lambda_2}(x)\ge\varphi_{\lambda_1}(x)
 \quad\text{for every }x\in {M_0}.
\end{equation}
The family constructed in Remark~\ref{exloc} has this property.  We deform
only the Riemannian tensor on $M_0$, so that 
\begin{equation}\label{causaldeformation}
 g_\lambda=
-\Lambda\,dt^2+ \omega\otimes dt+dt\otimes\omega+\varphi_\lambda g_0.
\end{equation}
The associated Randers--Kropina metric is
\begin{equation}\label{RKlambdametric}
 F_\lambda(x,v)=
 \frac{\varphi_\lambda(x)g_0(v,v)}{-\omega(v)+
 \sqrt{\Lambda(x)\varphi_\lambda(x)g_0(v,v)+\omega(v)^2}},
\end{equation}
with the same conic domain~\eqref{RKdomain}.  In particular, on
$\{\Lambda=0\}$,
\begin{equation}\label{Kropinalambda}
 F_\lambda(x,v)=-\frac{\varphi_\lambda(x)g_0(v,v)}{2\omega(v)}.
\end{equation}
Following \cite{bartolocaponio,cjmRK}, we approximate \eqref{causaldeformation}  by the standard stationary metrics
\begin{equation*}
 g_{\lambda,\epsilon}= -(\Lambda+\epsilon)dt^2+
 \omega\otimes dt+dt\otimes\omega+\varphi_\lambda g_0
, \qquad \epsilon>0,
\end{equation*}
whose Fermat metrics are the genuine Randers metrics
\begin{equation}\label{RKapprox}
 F_{\lambda,\epsilon}(x,v)=
 \frac{\sqrt{(\Lambda(x)+\epsilon)\varphi_\lambda(x)g_0(v,v)
 +\omega(v)^2}+\omega(v)}{\Lambda(x)+\epsilon}.
\end{equation}
We write $F_\epsilon$ for the approximating Randers metric obtained from
$g_0$, $\omega$ and $\Lambda$, namely, for the metric obtained by setting
$\varphi_\lambda\equiv1$ in~\eqref{RKapprox}.  We shall use the
following compactness assumption, introduced in \cite{cjs}:
\refstepcounter{equation}
\begin{equation*}\tag{$\star$}\label{star}
\begin{aligned}
 &\text{there exists $\bar\epsilon\in(0,1]$ such that every intersection of a closed}
 \\[-2pt]
 &\text{forward ball and a closed backward ball of $F_{\bar\epsilon}$ is compact.}
\end{aligned}
\end{equation*}
For $0<\epsilon\le\bar\epsilon$ and every $\lambda>0$ we have
\begin{equation}\label{RKmonotonicities}
	F_{\lambda,\epsilon}\ge F_\epsilon\ge F_{\bar\epsilon},
	\qquad
	F_{\lambda,\epsilon}<F_\lambda
	\quad\text{on }\mathcal A.
\end{equation}
Indeed, the inequality $F_\epsilon\ge F_{\bar\epsilon}$ and the strict one
follow from the monotonicity of the Randers approximation
\cite[Proposition~3.1]{cjmRK}, while $F_{\lambda,\epsilon}\ge F_\epsilon$ follows
from $\varphi_\lambda\ge1$.  By
\eqref{monotone-localized} and \eqref{RKapprox} we also have
$F_{\lambda_2,\epsilon}\ge F_{\lambda_1,\epsilon}$ for every
$\lambda_2\ge\lambda_1$ and every $0<\epsilon\le\bar\epsilon$.
Consequently, condition~\eqref{star} also gives the compactness of the
intersections of closed forward and backward balls for every
$F_{\lambda,\epsilon}$, $0<\epsilon\le\bar\epsilon$.  Throughout this
subsection, we extend $F$ and  $F_\lambda$ to the
zero section by setting their value equal to $0$.  
\begin{remark}\label{PStar}
We remark that
condition \eqref{star} is equivalent to global hyperbolicity of
$(\R\times M_0,g_{\bar \eps})$ by \cite[Proposition 2.2 and Theorem
4.3]{cjs}.  Moreover, it implies that the energy functional of
$F_{\bar\eps}$ (and thus of each $F_{\lambda,\eps}$ with
$0<\eps\le\bar\eps$) satisfies the Palais--Smale condition on
$\Omega_{p,q}(M_0)$ (see the comments above \cite[Theorem 5.2]{cjs}).
\end{remark}
For a subset $U\subset {M_0}$, let
\begin{equation*}
 \mathscr A_{p,q}(U)=
 \left\{x\in\Omega_{p,q}(U):
 \dot x(s)\in\mathcal A_{x(s)}\cup\{0\}\ \text{for a.e. }s\in [0,1],
 \ \int_0^1F^2(x,\dot x)\,ds<+\infty\right\}.
\end{equation*}
We shall need the following localized form of the standard finite category
property of energy sublevels in a Finsler manifold.  
The proof is essentially the usual one for
Riemannian path spaces, see
\cite[\S16]{milnor} (note that the broken geodesics used below are those of the
auxiliary Riemannian metric $h_R$, the Finsler metric $G$ entering only through
the energy bound~{\rm (c)}); the only additional point is that the category
is computed in $\Omega_{p,q}(M_0\setminus\Sigma)$.  The uniform positive
distance from the closed set $\Sigma$ allows the entire
finite-dimensional reduction to be carried out inside
$M_0\setminus\Sigma$.
\begin{lemma}\label{finitecategorysublevel}
Let $G$ be a (regular) Finsler metric on $M_0$, let $\Sigma\subset M_0$ be closed, and let $p,q\in M_0\setminus\Sigma$.  Let $h_R$ be a complete 	auxiliary Riemannian metric on $M_0$ and assume that 	$\mathcal Q\subset\Omega_{p,q}(M_0\setminus\Sigma)$ satisfies the following 	properties:
\begin{enumerate}[label=(\alph*)]
\item the images of all curves in $\mathcal Q$ are contained in a common compact set $K\subset M_0$;
\item there exists $\delta>0$ such that	$d_{h_R}\bigl(x([0,1]),\Sigma\bigr)\ge\delta$ for every $x\in\mathcal Q$;
\item $\sup_{x\in\mathcal Q}\int_0^1G^2(x,\dot x)\,ds<+\infty$.
\end{enumerate}
Then
\[
\cat_{\Omega_{p,q}(M_0\setminus\Sigma)}(\mathcal Q)<+\infty.
\]
\end{lemma}
\begin{proof}
Let us set $X:=\Omega_{p,q}(M_0\setminus\Sigma)$,	which is a Hilbert manifold, since $M_0\setminus\Sigma$ is open in $M_0$.
Replacing $K$ by $K\cap\{y\in M_0:d_{h_R}(y,\Sigma)\ge\delta\}$, which is	still compact and, by~(a) and~(b), still contains the image of every curve 	of $\mathcal Q$, we can assume that $K\subset\{y\in M_0:\ d_{h_R}(y,\Sigma)\ge\delta\}$. 	Fix $\eta\in(0,\delta/2)$ and put
$U:=\{y\in M_0:\ d_{h_R}(y,K)<\eta\}$.	Then $U$ is open and $K\subset U$; in particular $p,q\in U$.  Moreover	$d_{h_R}(y,\Sigma)>\delta-\eta>\delta/2$ for every $y\in U$, so that $U\subset M_0\setminus\Sigma$,	and $\overline U$ is compact, being closed and bounded in the complete Riemannian manifold $(M_0,h_R)$.  Since $M_0\setminus\Sigma$ is open,	$\Omega_{p,q}(U)$ is an open subset of $X$; consequently, every subset of	$\Omega_{p,q}(U)$ which is contractible in $\Omega_{p,q}(U)$ is a fortiori contractible in $X$.

Let $\rho\in(0,\eta/2)$ be such that, for every $z\in K$, the $h_R$-ball	$B_{h_R}(z,2\rho)$ is strongly convex and $\exp^{h_R}_z$ is a diffeomorphism	of the ball of radius $2\rho$ of $T_zM_0$ onto $B_{h_R}(z,2\rho)$; such a	$\rho$ exists because $K$ is compact.  Since $\rho<\eta/2$, all	these balls are contained in $U$.	Since $G$ and $\sqrt{h_R}$ are uniformly equivalent on $TM_0|_K$, by~(c)	there exists $A>0$ such that
\[
\int_0^1|\dot x|_{h_R}^2\,ds\le A,\qquad\text{for every }x\in\mathcal Q .
\]
Hence, by the Cauchy--Schwarz inequality,	
\[
d_{h_R}(x(s),x(t))\le\sqrt A\,|t-s|^{1/2},	\qquad s,t\in[0,1],\ x\in\mathcal Q .
\]
Choose a partition $0=t_0<t_1<\cdots<t_\nu=1$ such that
\[
\sqrt A\,(t_j-t_{j-1})^{1/2}<\frac\rho4,\qquad j=1,\ldots,\nu .
\]
For $x\in\mathcal Q$ and $j=1,\ldots,\nu$ set
\[
z_j(x):=x(t_{j-1}),\qquad	w_j(x):=x(t_j),\qquad	\theta_j(t):=\frac{t-t_{j-1}}{t_j-t_{j-1}} .
\]
Thus  $z_j(x),w_j(x)\in K$ and, by the choice of the partition, $d_{h_R}\bigl(z_j(x),x(t)\bigr)<\rho/4$ for every $t\in[t_{j-1},t_j]$.	Therefore the vectors
$\xi_{x,j}(t):=\bigl(\exp^{h_R}_{z_j(x)}\bigr)^{-1}\bigl(x(t)\bigr)$
are well defined and satisfy $|\xi_{x,j}(t)|_{h_R}<\rho/4$.  For 	$\tau\in[0,1]$, $j=1,\ldots,\nu$ and $t\in[t_{j-1},t_j]$, define
\[H(x,\tau)(t):=	\exp^{h_R}_{z_j(x)}	\Bigl((1-\tau)\,\xi_{x,j}(t)						+\tau\,\theta_j(t)\,\xi_{x,j}(t_j)	\Bigr).
\]
Both $\xi_{x,j}(t)$ and $\theta_j(t)\xi_{x,j}(t_j)$ have $h_R$-norm smaller	than $\rho/4$, hence so does any of their convex combinations; since the	corresponding tangent ball is contained in the domain of injectivity of	$\exp^{h_R}_{z_j(x)}$, the above expression is well defined and
\[
H(x,\tau)(t)\in B_{h_R}\bigl(z_j(x),\rho/4\bigr)\subset U .
\]
Because $\xi_{x,j}(t_{j-1})=0$ and $\theta_j(t_{j-1})=0$ we get $H(x,\tau)(t_{j-1})=z_j(x)=x(t_{j-1})$, while $\theta_j(t_j)=1$ gives	$H(x,\tau)(t_j)=w_j(x)=x(t_j)$.  Thus the definitions on consecutive intervals agree at the partition points and the endpoints $x(0)=p$ and	$x(1)=q$ are preserved, so that
$H:\mathcal Q\times[0,1]\longrightarrow\Omega_{p,q}(U)$
is well defined and satisfies $H(x,0)=x$ and $H(x,1)=P(x)$, where $P(x)$ 	denotes the broken $h_R$-geodesic which, on each interval $[t_{j-1},t_j]$,	joins $x(t_{j-1})$ to $x(t_j)$.  The continuity of $H$	in the $W^{1,2}$ topology follows from the continuity of the evaluation maps	$x\mapsto x(t_j)$ and from the smooth dependence of the exponential map and	of its local inverse on their arguments.
Indeed, all arguments remain in compact subsets of a common normal neighborhood, so the relevant maps and their first derivatives are uniformly bounded there, and the associated superposition operators are continuous on $W^{1,2}$. In particular $P=H(\cdot,1)$ is continuous.	The curve $P(x)$ depends only on the vertex tuple
$\mathcal V(x):=\bigl(x(t_1),\ldots,x(t_{\nu-1})\bigr)\in K^{\nu-1}$;
more precisely, there is a continuous map $\Pi$, defined on the set of	tuples $(z_1,\ldots,z_{\nu-1})\in K^{\nu-1}$ satisfying	$d_{h_R}(z_{j-1},z_j)\le\rho/4$ for $j=1,\ldots,\nu$, where $z_0:=p$ and	$z_\nu:=q$, such that $P=\Pi\circ\mathcal V$.  Let $\mathcal Z$ be the	closure of $\mathcal V(\mathcal Q)$ in $K^{\nu-1}$.  Then $\mathcal Z$ is	compact and, the above constraints being closed, $\Pi$ is defined on all of	$\mathcal Z$.  Hence
$\mathcal P:=\Pi(\mathcal Z)$ is a compact subset of $\Omega_{p,q}(U)$ containing $P(\mathcal Q)$. Since $\Omega_{p,q}(U)$ is a Hilbert manifold, every $y\in\mathcal P$ admits	open neighborhoods
\[
y\in W_y\subset\overline{W_y}\subset O_y\subset\Omega_{p,q}(U)
\]
with $O_y$ contractible in $\Omega_{p,q}(U)$, hence contractible in $X$.  By compactness there exist $y_1,\ldots,y_\ell\in\mathcal P$ such	that $\mathcal P\subset\bigcup_{i=1}^{\ell}W_{y_i}$.  Set
\[
\mathcal C_i:=\mathcal P\cap\overline{W_{y_i}},	\qquad	\mathcal Q_i:=P^{-1}(\mathcal C_i)\subset\mathcal Q,	\qquad i=1,\ldots,\ell .
\]
The sets $\mathcal C_i$ are closed in $\mathcal P$, they cover $\mathcal P$	and satisfy $\mathcal C_i\subset O_{y_i}$; hence each $\mathcal C_i$ is	contractible in $X$.  Since $P$ is continuous, the sets $\mathcal Q_i$ are	closed in $\mathcal Q$ and cover $\mathcal Q$.  Finally, the restriction of	$H$ to $\mathcal Q_i\times[0,1]$ is a homotopy taking values in	$\Omega_{p,q}(U)\subset X$ between the inclusion	$\mathcal Q_i\hookrightarrow X$ and the map
$P|_{\mathcal Q_i}:\mathcal Q_i\longrightarrow\mathcal C_i\subset X$,
and the latter is homotopic in $X$ to a constant map, because	$\mathcal C_i$ is contractible in $X$.  Therefore each $\mathcal Q_i$ is	contractible in $X$, and consequently
$\cat_X(\mathcal Q)\le\ell<+\infty$.
\end{proof}
In contrast with the regular Finsler case, one cannot obtain the required
minimax families merely by retracting arbitrary path-space subsets, since
such a retraction need not preserve the admissible cone.  We therefore
introduce the following condition, which gives exactly the existence of
compact admissible families of arbitrarily large category, finite
$F$-energy, and support outside the perturbed region.
\begin{definition}\label{categorycond}
The localized deformation datum satisfies the \emph{admissible category
condition} for $F$, $p$ and $q$ if, for every $k\in\mathbb N$, there exists a set
\[
 K_k\subset\mathscr A_{p,q}({M_0}\setminus B_3)
\]
which is compact in the $W^{1,2}$ topology and such that
\begin{equation*}
 \cat_{\Omega_{p,q}({M_0}\setminus B_1)}(K_k)\ge k,
 \qquad
 C_k:=\sup_{x\in K_k}\int_0^1F^2(x,\dot x)\,ds<+\infty.
\end{equation*}
\end{definition}
\begin{remark}\label{enoughRK}
A sufficient criterion for the admissible category condition follows from	the control construction in~\cite{cjmRKcorr} and the homotopy theorem	of~\cite{cms}. Suppose that there exists a connected, non-contractible	open set
$U\subset M_0\setminus B_3$, $p,q\in U$,	such that the inclusion
$i:\Omega_{p,q}(U)	\hookrightarrow	\Omega_{p,q}(M_0\setminus B_1)$
is a homotopy equivalence. Assume moreover that $\omega$ is nowhere zero	on $U$, that
$\mathcal D=\ker\omega|_U$
is completely non-holonomic of step two, and that the uniform boundedness and bracket assumptions used in
\cite{cms,cjmRKcorr} hold on $U$.	Let $\mathscr C_{p,q}(U)\subset\mathscr A_{p,q}(U)$
be the space of $W^{1,2}$ trajectories generated by the admissible 	controls introduced in~\cite{cms}, endowed with the topology	induced by $W^{1,2}([0,1],U)$. By	\cite[Theorem 3]{cjmRKcorr}, which is based on
\cite[Theorem~3.2]{cms}, the inclusion
$j:\mathscr C_{p,q}(U)\hookrightarrow\Omega_{p,q}(U)$ is a homotopy equivalence.	By the theorem of Fadell--Husseini, for every $k\ge1$ there exists a	compact set
$L_k\subset\Omega_{p,q}(U)$
such that $	\cat_{\Omega_{p,q}(U)}(L_k)\ge k$.
Let $f:\Omega_{p,q}(U)\longrightarrow\mathscr C_{p,q}(U)$ be a homotopy inverse of $j$, and set
$K_k:=j(f(L_k))$. Then $K_k$ is compact in the $W^{1,2}$ topology and 
$K_k\subset\mathscr A_{p,q}(U)	\subset\mathscr A_{p,q}(M_0\setminus B_3)$.
Since
$j\circ f\simeq\operatorname{id}_{\Omega_{p,q}(U)}$,	Remark~4 of~\cite{cjmRKcorr} gives
$\cat_{\Omega_{p,q}(U)}(K_k)\ge	\cat_{\Omega_{p,q}(U)}(L_k)	\ge k$.
Since the inclusion $i$ is a homotopy equivalence, it follows that
$\cat_{\Omega_{p,q}(M_0\setminus B_1)}(K_k)\ge k$.
Moreover, \cite[Proposition~5]{cjmRKcorr}, applied to the compact	set $L_k$, gives
\[
\sup_{x\in K_k}	\int_0^1F^2(x,\dot x)\,ds	<+\infty.
\]
Thus the admissible category condition holds.

If $U$ is relatively compact and the step-two condition holds on a neighborhood of $\overline U$, the uniform bounds on $\omega$ and on	the transverse brackets required above can be obtained from a finite	covering of $\overline U$ (see \cite[Remark 2]{cjmRKcorr}).
\end{remark}

We can now state the localized multiplicity result.
\begin{theorem}\label{mainRK}
Let $g$ be a metric on $\R\times M_0$ as in  \eqref{causalpartialt}, and assume
condition~\eqref{star}.  Let $p,q\in {M_0}$, $p\ne q$, and let
$B_1\subset B_2\subset B_3$, $(\varphi_\lambda)_{\lambda>0}$ be a localized
deformation datum adapted to $p$ and $q$, with $\overline{B_2}$ compact, $(\varphi_\lambda)_{\lambda>0}$ satisfying~\eqref{monotone-localized} and with condition~\eqref{dab} interpreted using the distance $d_0$ induced
by $g_0$.  Assume that the admissible category
condition of Definition~\ref{categorycond} holds.
Then, for every $m\in\mathbb N$ and any $t_0\in\R$, there exists
$\lambda_0=\lambda_0(p,q,m)>0$ such that, for every $\lambda>\lambda_0$, the
Randers--Kropina metric $F_\lambda$ in~\eqref{RKlambdametric} admits at least
$m$ geodesics joining $p$ to $q$, with distinct $F_\lambda$-lengths,
and with images contained in ${M_0}\setminus B_1$.  Equivalently, the spacetime
$(\R\times {M_0},g_\lambda)$ admits at least $m$ future-pointing light rays
from $(t_0,p)$ to the line $\mathbb R\times \{q\}$, whose spatial
projections avoid $B_1$.
\end{theorem}
\begin{proof}
Throughout the proof we set
\[
E_{\lambda,\epsilon}(x):=\int_0^1F_{\lambda,\epsilon}^2(x,\dot x)\,ds,	\qquad x\in\Omega_{p,q}(M_0),
\]
and
\[
L_{12}:=\max_{\overline{B_2}}(\Lambda+\bar\epsilon),\qquad M_{12}:=\max_{\overline{B_2}}\|\omega\|_{g_0}.
\]
Since $\overline{B_2}$ is compact, $\overline{B_1}\subset B_2$ and we may choose a closed set $B'_1$ such that
\begin{equation*}
\overline{B_1}\subset\operatorname{int}B'_1\subset B'_1\subset B_2,		\qquad		\delta_0:=d_0(B'_1,\partial B_2)>0.
\end{equation*}
Let $x\in B_2\setminus B_1$, $v\in T_xM_0\setminus\{0\}$, and
$0<\epsilon\le\bar\epsilon$.  Rationalizing the expression in
\eqref{RKapprox}, we obtain
\begin{align*}
F_{\lambda,\epsilon}(x,v)&=\frac{\varphi_\lambda(x)g_0(v,v)}{\sqrt{(\Lambda(x)+\epsilon)\varphi_\lambda(x)g_0(v,v)+\omega(v)^2}-\omega(v)}.
\end{align*}
Since $|\omega(v)|\le M_{12}\sqrt{g_0(v,v)}$ and	$\Lambda(x)+\epsilon\le L_{12}$, it follows that
\begin{align*}
F_{\lambda,\epsilon}(x,v)&\ge\frac{\varphi_\lambda(x)}{\sqrt{L_{12}\varphi_\lambda(x)+M_{12}^2}+M_{12}}
\sqrt{g_0(v,v)}	\ge a_\lambda\sqrt{g_0(v,v)},
\end{align*}
where
\[
a_\lambda:=\frac{I_\lambda}{\sqrt{L_{12}I_\lambda+M_{12}^2}+M_{12}}	\longrightarrow
+\infty\quad\text{as }\lambda\to+\infty.
\]
If a curve $x\in\Omega_{p,q}(M_0)$ meets $B'_1$ then, since
$p,q\in M_0\setminus B_3$, its first entrance into $B'_1$, together with the
preceding crossing of $\partial B_2$, determines an arc joining
$\partial B_2$ to $\partial B'_1$ whose interior is contained in
$B_2\setminus B'_1\subset B_2\setminus B_1$.  Hence, by the Cauchy--Schwarz
inequality,
\begin{equation}\label{RKbarrier}
E_{\lambda,\epsilon}(x)\ge a_\lambda^2\delta_0^2,
\end{equation}
uniformly for $0<\epsilon\le\bar\epsilon$.  Let us define	$b(\lambda):=\frac{1}{2} a_\lambda^2\delta_0^2$.  Then $b(\lambda)\to+\infty$.	Moreover, $I_\lambda$ is non-decreasing by \eqref{monotone-localized}; hence
$a_\lambda$ and $b(\lambda)$ are non-decreasing.  In particular, every curve in the $b(\lambda)$-sublevel of $E_{\lambda,\epsilon}$ avoids $B'_1$, for	every $0<\epsilon\le\bar\epsilon$.
We next obtain a uniform upper estimate on the compact admissible families provided by Definition~\ref{categorycond}.  Let $v\in\mathcal A\cup\{0\}$ be	based at a point of $M_0\setminus B_3$.  Using \eqref{RKmetric}, the second inequality in \eqref{RKmonotonicities} and  $\varphi_\lambda\ge1$ on $M_0$,  $\varphi_\lambda\le S$ on	$M_0\setminus B_3$, we get 	
\begin{equation*}
F_{\lambda,\epsilon}(v)\le F_\lambda(v)\le\varphi_\lambda F(v)\le S F(v).
\end{equation*}
Consequently, for every compact set $K_k$ appearing in	Definition~\ref{categorycond},	\begin{equation}\label{RKupperK}
\sup_{x\in K_k}E_{\lambda,\epsilon}(x)		\le S^2\sup_{x\in K_k}\int_0^1F^2(x,\dot x)\,ds=S^2C_k
\end{equation}
for every $\lambda>0$ and every $0<\epsilon\le\bar\epsilon$.	 Set now
$\alpha_0:=0$, $k_1:=1$ and suppose  $k_h$ has been chosen.  Let $K_{k_h}$ be supplied by the
admissible category condition.  Let us choose
$\alpha_h>\max\{\alpha_{h-1},S^2C_{k_h}\}$	and then $\lambda_h>0$, with $\lambda_h>\lambda_{h-1}$  so large	that $b(\lambda_h)>\alpha_h$.
Since $b$ is non-decreasing, this inequality remains true for every $\lambda\ge\lambda_h$.

Let us consider the sublevel
\[
\mathcal Q_h:=
\left\{
x\in\Omega_{p,q}(M_0):
E_{\lambda_h,\bar\epsilon}(x)\leq\alpha_h
\right\}.
\]
We claim that every curve in $\mathcal Q_h$ avoids $B'_1$.
Indeed, suppose by contradiction that $x\in\mathcal Q_h$ and that $x([0,1])\cap B'_1\neq\varnothing$.
Then \eqref{RKbarrier}, applied with $\lambda=\lambda_h$ and $\epsilon=\bar\epsilon$, gives
$E_{\lambda_h,\bar\epsilon}(x) \geq a_{\lambda_h}^2\delta_0^2$.
On the other hand, by the choice of $\lambda_h$,
$\alpha_h<b(\lambda_h)=\frac{1}{2}a_{\lambda_h}^2\delta_0^2 <a_{\lambda_h}^2\delta_0^2$,
and therefore 
\[
a_{\lambda_h}^2\delta_0^2
\leq E_{\lambda_h,\bar\epsilon}(x)
\leq\alpha_h
<a_{\lambda_h}^2\delta_0^2,
\]
which is impossible.  Hence, $x([0,1])\cap B'_1=\varnothing$, for every $x\in\mathcal Q_h$.
Thus,
\[
\mathcal Q_h\subset\Omega_{p,q}(M_0\setminus B'_1)\subset\Omega_{p,q}(M_0\setminus B_1).
\]
We claim that all curves in $\mathcal Q_h$ have images contained in a common	compact subset of $M_0$.  Indeed, for every $x\in\mathcal Q_h$ and every	$s\in[0,1]$,
\begin{align*}
d^+_{\lambda_h,\bar\epsilon}(p,x(s))&\le\int_0^sF_{\lambda_h,\bar\epsilon}(x,\dot x)\,dr\le\sqrt{\alpha_h},\\
d^+_{\lambda_h,\bar\epsilon}(x(s),q)&\le\int_s^1F_{\lambda_h,\bar\epsilon}(x,\dot x)\,dr\le\sqrt{\alpha_h}.
\end{align*}
Thus
\[
x([0,1])\subset	\overline B^+_{\lambda_h,\bar\epsilon}(p,\sqrt{\alpha_h})\cap
\overline B^-_{\lambda_h,\bar\epsilon}(q,\sqrt{\alpha_h}).
\]
By \eqref{RKmonotonicities}, $F_{\lambda_h,\bar\epsilon}\ge F_{\bar\epsilon}$,	and hence the preceding intersection is contained in
\[
\overline B^+_{\bar\epsilon}(p,\sqrt{\alpha_h})\cap\overline B^-_{\bar\epsilon}(q,\sqrt{\alpha_h}),
\]
which is compact by condition~\eqref{star}.  Therefore the curves in	$\mathcal Q_h$ remain in a common compact subset of $M_0$ and avoid $B'_1$; in particular, 
 they stay at a $d_0$-distance from $B_1$ at least $d_0\bigl(B_1,M_0\setminus\operatorname{int}B'_1\bigr)>0$, which is    positive because $B_1$ is compact and disjoint from the closed set      $M_0\setminus\operatorname{int}B'_1$.       Lemma~\ref{finitecategorysublevel}, applied with       $G=F_{\lambda_h,\bar\epsilon}$ and $\Sigma=B_1$   therefore gives $N_h:=\cat_{\Omega_{p,q}(M_0\setminus B_1)}(\mathcal Q_h)<+\infty$. If $h<m$, we  choose	$k_{h+1}>\max\{k_h,N_h\}$ and repeat the construction.  We obtain integers $1=k_1<k_2<\cdots<k_m$, numbers $0=\alpha_0<\alpha_1<\cdots<\alpha_m$, with $S^2C_{k_h}<\alpha_h$, and parameters $0<\lambda_1<\cdots<\lambda_m$ such that $b(\lambda_h)>\alpha_h$, for every $h=1,\ldots,m$. For $k\ge1$, let $\Gamma_k$ be the family of subsets	$A\subset\Omega_{p,q}(M_0\setminus B_1)$ satisfying $\cat_{\Omega_{p,q}(M_0\setminus B_1)}(A)\ge k$, and let us define
\[
c_{\lambda,\epsilon,k}:=\inf_{A\in\Gamma_k}\sup_{x\in A}	E_{\lambda,\epsilon}(x).
\]
Let $\lambda\ge\lambda_m$, $0<\epsilon\le\bar\epsilon$.
For $h=1$, the lower bound $c_{\lambda,\epsilon,k_1}\ge\alpha_0=0$ is	immediate. 
Let now $h\geq2$. Since
$\lambda\geq\lambda_m\geq\lambda_{h-1}$, the pointwise monotonicity
of the deformation gives $F_{\lambda,\epsilon}\geq F_{\lambda_{h-1},\epsilon}$.
On the other hand, applying the monotonicity of the Randers
approximation to the fixed data $(\varphi_{\lambda_{h-1}}g_0,\omega,\Lambda)$, the inequality
$0<\epsilon\leq\bar\epsilon$ yields $F_{\lambda_{h-1},\epsilon}\geq F_{\lambda_{h-1},\bar\epsilon}$.
Consequently, $E_{\lambda,\epsilon}(x)\geq E_{\lambda_{h-1},\bar\epsilon}(x)$, for every $x\in\Omega_{p,q}(M_0)$.
Therefore, if $E_{\lambda,\epsilon}(x)\leq\alpha_{h-1}$, then
$E_{\lambda_{h-1},\bar\epsilon}(x)\leq E_{\lambda,\epsilon}(x)\leq\alpha_{h-1}$,
and hence, by the definition of $\mathcal Q_{h-1}$,
$x\in\mathcal Q_{h-1}$. Thus,
\[
\left\{E_{\lambda,\epsilon}\leq\alpha_{h-1}\right\}\subset\mathcal Q_{h-1}.
\]
Since	$\cat_{\Omega_{p,q}(M_0\setminus B_1)}(\mathcal Q_{h-1})=N_{h-1}<k_h$,
no set belonging to $\Gamma_{k_h}$ can be contained in the	$\alpha_{h-1}$-sublevel of $E_{\lambda,\epsilon}$. Hence $c_{\lambda,\epsilon,k_h}\ge\alpha_{h-1}$.
On the other hand,
\[
K_{k_h}\subset\mathscr A_{p,q}(M_0\setminus B_3)\subset	\Omega_{p,q}(M_0\setminus B_1)
\]
and $\cat_{\Omega_{p,q}(M_0\setminus B_1)}(K_{k_h})\ge k_h$.
Thus $K_{k_h}\in\Gamma_{k_h}$, and \eqref{RKupperK} gives $c_{\lambda,\epsilon,k_h}\le S^2C_{k_h}$.
Since $\lambda\ge\lambda_m\ge\lambda_h$ and $b$ is non-decreasing,	$\alpha_h<b(\lambda)$.  We have therefore proved the pinching estimates
\begin{equation}\label{RKpinching}
	\alpha_{h-1}\le c_{\lambda,\epsilon,k_h}\le S^2C_{k_h}<\alpha_h<b(\lambda),		\qquad h=1,\ldots,m.
\end{equation}
By Remark~\ref{PStar}, $E_{\lambda,\epsilon}$	satisfies the Palais--Smale condition on $\Omega_{p,q}(M_0)$.  We claim	that every value $c_{\lambda,\epsilon,k_h}$ is critical.  Indeed, suppose	that $c:=c_{\lambda,\epsilon,k_h}$ is not a critical value.  By
\eqref{RKpinching}, one can choose $\eta>0$ such that $c+2\eta<b(\lambda)$	and such that the interval $[c-2\eta,c+2\eta]$ contains no critical value.	The  Deformation Lemma then 
lowers the $(c+\eta)$-sublevel into the $(c-\eta)$-sublevel.  Since every level involved is below $b(\lambda)$, the	barrier estimate implies that every curve occurring during this deformation
avoids $B'_1$.  Thus the deformation is entirely contained in	$\Omega_{p,q}(M_0\setminus B_1)$ and, exactly as in the proof of	Theorem~\ref{main}, it does not decrease the Lusternik--Schnirelmann	category relative to $\Omega_{p,q}(M_0\setminus B_1)$.  This contradicts	the definition of $c_{\lambda,\epsilon,k_h}$.  Therefore
$c_{\lambda,\epsilon,k_h}$ is a critical value.

For every $h=1,\ldots,m$, choose a corresponding critical point	$x_{\epsilon,h}$.  It is an $F_{\lambda,\epsilon}$-geodesic from $p$ to $q$	and satisfies
\begin{equation}\label{ineqh}
	\alpha_{h-1}\le E_{\lambda,\epsilon}(x_{\epsilon,h})\le S^2C_{k_h}<\alpha_h.
\end{equation}
The energy levels $ E_{\lambda,\epsilon}(x_{\epsilon,h})$ are then pairwise distinct, since $S^2C_{k_h}<\alpha_h$, whereas	$\alpha_h$ is the lower endpoint of the inequalities with $h+1$ as in \eqref{ineqh}.  Consequently, the	geodesics $x_{\epsilon,1},\ldots,x_{\epsilon,m}$ are pairwise distinct and avoid $B'_1$.	Fix now $\lambda\ge\lambda_m$ and let $\epsilon_n\downarrow0$.  For each	$h$, the energies $E_{\lambda,\epsilon_n}(x_{\epsilon_n,h})$ are uniformly	bounded by $S^2C_{k_h}$.  The convergence theorem for Randers approximations	of a Randers--Kropina metric, \cite[Proposition~3.9]{cjmRK}, applies to the	fixed data $(\varphi_\lambda g_0,\omega,\Lambda)$ (recall that condition	\eqref{star} holds for these data with the same $\bar\epsilon$, since the	closed balls of $F_{\lambda,\bar\epsilon}$ are contained in those of	$F_{\bar\epsilon}$ by \eqref{RKmonotonicities}).  Since $p\ne q$, the	limiting spatial curves are nonconstant.  Passing to a subsequence,  the	curves $x_{\epsilon_n,h}$ converge uniformly, for every $h=1,\ldots,m$, to	an $F_\lambda$-geodesic $x_h\colon[0,1]\to M_0$ joining $p$ to $q$, and $\ell_{F_{\lambda,\epsilon_n}}(x_{\epsilon_n,h})\longrightarrow	\ell_{F_\lambda}(x_h)$.
The approximating geodesics are parametrized with constant	$F_{\lambda,\epsilon_n}$-speed on $[0,1]$. Therefore $E_{\lambda,\epsilon_n}(x_{\epsilon_n,h})	=\ell_{F_{\lambda,\epsilon_n}}(x_{\epsilon_n,h})^2$.
Passing to the limit in \eqref{ineqh}, we obtain
\[
\sqrt{\alpha_{h-1}}\le\ell_{F_\lambda}(x_h)\le S\sqrt{C_{k_h}}	<\sqrt{\alpha_h}.
\]
These intervals are pairwise disjoint, and hence the limiting	$F_\lambda$-geodesics have pairwise distinct lengths.
Moreover, they  are contained in $M_0\setminus \mathrm{int}(B'_1)$, thus they avoid $B_1$.	

Finally, by the Fermat principle in $(\R\times M_0, g_\lambda)$,	each $F_\lambda$-geodesic $x_h$ lifts to a future-pointing null pregeodesic of	$g_\lambda$ starting at $(t_0,p)$ and ending at	$\bigl(t_0+\ell_{F_\lambda}(x_h),q\bigr)\in\mathbb R\times\{q\}$.  The	corresponding arrival times are pairwise distinct because the	$F_\lambda$-lengths are pairwise distinct.  Taking $\lambda_0:=\lambda_m$
concludes the proof.
\end{proof}

\begin{remark}
When $\Lambda\equiv0$, the Killing field is null everywhere and
Theorem~\ref{mainRK} is a localized multiplicity theorem for the Kropina
metrics~\eqref{Kropinalambda}.  Through Zermelo navigation, this is the {\em critical wind} case in \cite{cjswind}, in which the
indicatrix is tangent to the zero section; through the spacetime interpretation,
it gives multiple light rays for a localized deformation of the 
Riemannian tensor on $M_0$ while the  one-form $\omega$ in \eqref{causalpartialt} is kept fixed.  This is the
singular counterpart of Proposition~\ref{lightdeformation}.
\end{remark}
\begin{example}\label{RKexample}
We exhibit an explicit set of data, on a contractible $M_0$, satisfying all
the hypotheses of Theorem~\ref{mainRK}. Let $M_0=\R^3$ with the standard coordinates $x=(x^1,x^2,x^3)$, let $g_0$ be	the Euclidean metric and put
\[\omega:=dx^3-x^2\,dx^1 ,	\qquad	\Lambda(x):=\zeta(|x|),
\]
where $|\cdot|$ is the Euclidean norm and $\zeta\colon[0,+\infty)\to[0,1]$ is	smooth and non-decreasing, with $\zeta\equiv0$ on $[0,3]$, $\zeta\equiv1$ on	$[4,+\infty)$ and $0<\zeta<1$ on $(3,4)$.  Since	$\|\omega\|_{g_0}^2=1+(x^2)^2\ge1$, the inequality in \eqref{causalpartialt} holds and
\[
g=-\Lambda\,dt^2+dt\otimes\omega+\omega\otimes dt+g_0
\]
is a Lorentzian metric on $\R\times\R^3$ whose Killing vector field	$\partial_t$ is causal: it is lightlike on the closed ball $\{|x|\le3\}$ and	timelike outside it.  The associated future optical Lagrangian~\eqref{RKmetric}	is therefore a  Randers--Kropina metric, of Kropina type exactly on	$\{|x|\le 3\}$.  
	
\smallskip
\emph{Condition~\eqref{star}.}
Let $0<\bar\epsilon\le1$.  Rationalizing~\eqref{RKapprox} with	$\varphi_\lambda\equiv1$ and using $0\le\Lambda\le1$, $\bar\epsilon\le1$ and	$|\omega(v)|\le\sqrt{1+(x^2)^2}\,|v|$, we get, for every	$v\in T_xM_0\setminus\{0\}$,
\[
F_{\bar\epsilon}(x,v)	=\frac{g_0(v,v)}{\sqrt{(\Lambda(x)+\bar\epsilon)g_0(v,v)+\omega(v)^2}-\omega(v)}
	\ \ge\
	\frac{|v|}{\sqrt2+2\sqrt{1+(x^2)^2}}	=:\sqrt{\tilde g_x(v,v)} ,
\]
where $\tilde g:=\bigl(\sqrt2+2\sqrt{1+(x^2)^2}\bigr)^{-2}g_0$ is a Riemannian
metric conformal to $g_0$.  Since
\[
\inf_{|x|=\tau}\bigl(\sqrt2+2\sqrt{1+(x^2)^2}\bigr)^{-1}=\bigl(\sqrt2+2\sqrt{1+\tau^2}\bigr)^{-1}
	\quad\text{and}\quad	\int_0^{+\infty}\frac{d\tau}{\sqrt2+2\sqrt{1+\tau^2}}=+\infty,
\]
the metric $\tilde g$ is complete.  Hence every closed forward or backward	ball of $F_{\bar\epsilon}$ is a closed subset of a compact $\tilde g$-ball,	and is therefore compact.  Thus~\eqref{star} holds, for every	$\bar\epsilon\in(0,1]$.
	
	\smallskip
\emph{The localized deformation datum.}
Take the datum of Remark~\ref{exloc} with $M_0=\R^3$, the identity chart,	$x_0=0$ and radii $\rho_0=1,\rho_1=2, \rho_2=3,\rho_3=4$, so that	
\[	B_1=\{|x|\le2\},\qquad B_2=\{|x|<3\},\qquad B_3=\{|x|<4\},	\]
and $\varphi_\lambda=1+\lambda\,\eta(|x|)$ with $\eta$ as in	Remark~\ref{exloc}.  Then $\overline{B_2}$ is compact,	$d_0(B_1,\partial B_2)=1>0$, the family$(\varphi_\lambda)_{\lambda>0}$	satisfies~\eqref{monotone-localized} because $\eta\ge0$, and the datum is	strictly localized, the perturbation being supported in	$\{1\le|x|\le4\}$.  Finally, $M_0\setminus B_1=\{|x|>2\}$ is connected and	non-contractible, although $M_0$ is contractible.	Fix $r=5$, let $S:=\{|x|=r\}$ and let $p,q\in S$, $p\ne q$; then	$p,q\in M_0\setminus B_3$, so the datum is adapted to $p$ and $q$.
	
\smallskip
\emph{The admissible category condition.}
Let
\[
U:=\{x\in\R^3:\ 4<|x|<6\}\subset M_0\setminus B_3 ,
\]
an open, connected, non-contractible subset of $\R^3$ containing $S$, hence	$p$ and $q$.  Since $\Lambda\equiv1$ on $\overline U$, the metric $F$ is a	Randers metric on a neighborhood of $\overline U$; in particular	$\mathcal A_x=T_xM_0\setminus\{0\}$ for $x\in U$, so that
\[	\mathscr A_{p,q}(U)=\Omega_{p,q}(U),
\]
and, setting $M_U:=\max_{\overline U}\sqrt{1+(x^2)^2}$, the same	rationalization as above gives
\begin{equation}\label{RKexbound}
F(x,v)\le(1+2M_U)\,|v| ,\qquad x\in\overline U,\ v\in T_xM_0 .
\end{equation}
For $s\in[0,1]$ define $T_s\colon\{|x|>2\}\to\{|x|>2\}$ 
by
\[
T_s(x):=\bigl((1-s)|x|+s\,r\bigr)\frac{x}{|x|} .
\]
Exactly as in Remark~\ref{exloc}, the maps $T_s$ are locally Lipschitz,	locally uniformly in $s$, and induce a homotopy	$(\gamma,s)\mapsto T_s\circ\gamma$ which is continuous in the $H^1$ topology.	Since $|p|=|q|=r$, each $T_s$ fixes $p$ and $q$; moreover	$T_s\bigl(\{|x|>2\}\bigr)\subset\{|x|>2\}$ and, because $4<r<6$,	$T_s(U)\subset U$.  Hence $T$ induces deformation retractions of both	$\Omega_{p,q}(U)$ and $\Omega_{p,q}(M_0\setminus B_1)$ onto	$\Omega_{p,q}(S)$, and therefore the inclusion
\[
i\colon\Omega_{p,q}(U)\hookrightarrow\Omega_{p,q}(M_0\setminus B_1)
\]
is a homotopy equivalence. Now $U$ is a non-contractible open subset of $\R^3$, so	\cite[Proposition~3.2]{fh} applies to it  and provides, for every	$k\in\N$, a compact set $K_k\subset\Omega_{p,q}(U)$ with	$\cat_{\Omega_{p,q}(U)}(K_k)\ge k$.  Since $i$ is a homotopy equivalence,
\[
\cat_{\Omega_{p,q}(M_0\setminus B_1)}(K_k)\ge k .
\]
Finally, $K_k$ being compact in $\Omega_{p,q}(U)$, the union of the images of	its curves is a compact subset of $U$ and	$\sup_{x\in K_k}\int_0^1|\dot x|^2ds<+\infty$; hence, by~\eqref{RKexbound},
\[
C_k=\sup_{x\in K_k}\int_0^1F^2(x,\dot x)\,ds<+\infty .
\]
Thus the admissible category condition of Definition~\ref{categorycond} holds, the singular region 	$\{\Lambda=0\}$ being disjoint from $\overline U$.
	
\smallskip
\emph{Conclusion.}
All the hypotheses of Theorem~\ref{mainRK} are satisfied.  Consequently, for
every $m\in\N$ and every $t_0\in\R$ there is $\lambda_0>0$ such that, for
$\lambda>\lambda_0$, the Randers--Kropina metric
\[
F_\lambda(x,v)=	\frac{\varphi_\lambda(x)\,g_0(v,v)}{-\omega(v)+\sqrt{\Lambda(x)\varphi_\lambda(x)g_0(v,v)+\omega(v)^2}}
\]
admits at least $m$ geodesics from $p$ to $q$, with pairwise distinct	$F_\lambda$-lengths and images contained in $\{|x|>2\}$; equivalently, the	spacetime $\bigl(\R\times\R^3,g_\lambda\bigr)$, with $g_\lambda$ as	in~\eqref{causaldeformation}, admits at least $m$ future-pointing light rays	from $(t_0,p)$ to the line $\R\times\{q\}$, with pairwise distinct arrival	times and spatial projections avoiding the ball $\{|x|\le2\}$.  Since the	datum is strictly localized, $g_\lambda=g$ on $\R\times\{|x|>4\}$.
\end{example}
\begin{remark}
The same construction works in the purely Kropina case $\Lambda\equiv0$, in	which $\partial_t$ is lightlike everywhere and $F$ is the Kropina metric	$-g_0(v,v)/(2\omega(v))$ on $\{\omega(v)<0\}$.  Condition~\eqref{star} is	verified word for word as in Example~\ref{RKexample}, and the localized	deformation datum is unchanged.  The verification of the admissible category condition, however, is no longer elementary, since now	$\mathcal A_x=\{v:\omega(v)<0\}$ on all of $M_0$ and	$\mathscr A_{p,q}(U)\subsetneq\Omega_{p,q}(U)$.  One then applies the	sufficient criterion stated in Remark~\ref{enoughRK}: the one-form $\omega=dx^3-x^2dx^1$ is nowhere zero and	satisfies $\omega\wedge d\omega=dx^1\wedge dx^2\wedge dx^3\ne0$, so	$\mathcal D=\ker\omega|_U$ is a contact distribution, and in particular completely non-holonomic of	step two and it remains so also in the compact set $\bar U$.  Thus the required assumptions in \cite{cjmRKcorr} hold on $U$ (we recall \cite[Remark 2]{cjmRKcorr}).
\end{remark}
\section*{Acknowledgments}
	\noindent E.C. and A.M. are  partially supported   by  MUR under the Programme ``Department of Excellence'' Legge 232/2016  (Grant No. CUP - D93C23000100001).\\
All the authors  are partially supported  by  INdAM - GNAMPA and in particular E.C. and A.M. by  ``INdAM - GNAMPA Project''  CUP E53C25002010001.

\end{document}